\documentclass[a4paper]{amsart}

\usepackage[shortlabels]{enumitem}
\setlist[enumerate]{label={\arabic*.}}

\usepackage[dvipsnames]{xcolor}
\usepackage[backref]{hyperref}
\hypersetup{colorlinks=true,linkcolor=Maroon,citecolor=OliveGreen}
\usepackage{xfrac}
\usepackage{stmaryrd}

\usepackage{enumitem}
\setlist[description]{%
font={\normalfont\itshape}
}

\usepackage{pgfplots}
\pgfplotsset{compat=1.18}

\pgfplotsset{
  every axis/.append style={
    tick label style={font=\small},
    label style={font=\small},
    title style={font=\small},
  }
}

\usepackage{cleveref}

\usepackage[font=footnotesize]{caption}

\usepackage{microtype}

\newtheorem{theorem}{Theorem}[section]
\newtheorem*{theorem*}{Theorem}
\newtheorem{lemma}[theorem]{Lemma}
\newtheorem{proposition}[theorem]{Proposition}
\newtheorem{corollary}[theorem]{Corollary}
\newtheorem{requirement}{Requirement}

\theoremstyle{definition}

\newtheorem{remark}[theorem]{Remark}
\newtheorem*{remark*}{Remark}

\newtheorem*{conjecture*}{Conjecture}
\newtheorem{example}[theorem]{Example}
\newtheorem*{example*}{Example}

\usepackage{amssymb}
\usepackage{amsthm}
\usepackage{mathtools}
\usepackage{amsmath,amsfonts,amsthm,mathtools,commath,amssymb}

\usepackage{todonotes}

\usepackage{bbm}
\def\1{\mathbf{1}}
\def\F{\mathbf{F}}
\def\C{\mathbf{C}}
\def\R{\mathbf{R}}
\def\Z{\mathbf{Z}}

\def\P{\mathbf{P}}
\def\E{\mathbf{E}}

\def\unitary{\mathcal{U}}
\def\bounded{\mathcal{B}}
\def\tors{\mathrm{tors}}
\def\tf{\infty}
\def\Hcal{\mathcal{H}}
\def\pow{\mathrm{pow}}

\def\emb{\mathrm{emb}}
\def\grp{\mathrm{grp}}

\DeclareMathOperator{\tr}{Tr}
\DeclareMathOperator{\ntr}{tr}
\DeclareMathOperator{\fix}{fix}
\DeclareMathOperator{\supp}{supp}
\DeclareMathOperator{\Pref}{Pref}
\DeclareMathOperator{\FE}{FE}
\DeclareMathOperator{\Sym}{Sym}
\DeclareMathOperator{\Alt}{Alt}
\DeclareMathOperator{\std}{\mathrm{std}}

\DeclareMathOperator{\PSL}{PSL}
\DeclareMathOperator{\image}{im}
\DeclareMathOperator{\lcm}{lcm}

\begin{document}
\baselineskip=13pt 

\title[Strong convergence for free products]{Strong convergence of random representations of free products of finite groups}

\author{Marco Barbieri}
\address{Marco Barbieri, Faculty of Mathematics and Physics, University of Ljubljana, Jadranska 19, 1000 Ljubljana, Slovenia}
\email{marco.barbieri@fmf.uni-lj.si}

\author{Urban Jezernik}
\address{Urban Jezernik, Faculty of Mathematics and Physics, University of Ljubljana, Jadranska 19, 1000 Ljubljana, Slovenia / Institute of Mathematics, Physics, and Mechanics, Jadranska 19, 1000 Ljubljana, Slovenia}
\email{urban.jezernik@fmf.uni-lj.si}

\thanks{This work has been supported by the Slovenian Research Agency program P1-0222 and grants J1-70033, J1-50001, J1-4351.}

\begin{abstract}
We extend the polynomial method of Chen--Garza-Vargas--Tropp--van Handel and Magee--Puder--van Handel for operator-norm bounds in random permutation models to the setting where torsion is present. The main new feature is that asymptotic expansion of traces naturally involves fractional powers of $N$ rather than an ordinary Laurent series. We formulate fractional-power analogues of the method's key hypotheses and prove they lead to strong convergence. We verify these analogues for free products of finite groups $\Gamma=G_1*\cdots*G_m$. Concretely, for a uniformly random $\phi_N\in\hom(\Gamma,\Sym(N))$, set $\pi_N=\std\circ\phi_N$, where $\std$ denotes the standard $(N-1)$-dimensional representation of $\Sym(N)$ (the permutation representation with the trivial subrepresentation removed). We deduce strong convergence of $\pi_N$ to the left regular representation of $\Gamma$. As applications, we obtain asymptotically sharp spectral gaps for the associated random Schreier graphs, including almost Ramanujan behavior for $C_2*C_2*C_2$ and an explicit non-Ramanujan limiting spectral radius for $C_2*C_3\cong\PSL_2(\Z)$.
\end{abstract}

\maketitle


\section{Introduction}
\label{sec:intro}

\subsection{Expander graphs}


Expander graphs form a remarkable intersection of combinatorics, geometry, probability, and group theory. A family of finite connected graphs with an increasing number of vertices $N$, each of fixed degree $d$, is called an \emph{expander family} if each adjacency operator $\mathcal{A}$ in the family exhibits a uniform spectral gap between the top eigenvalue $d$ (with Perron-Frobenius eigenvector $\1$ of all ones) and the second-largest eigenvalue in absolute value, which is $\norm{\mathcal A |_{\1^\perp}}$ \cite{lubotzky1995cayley,kowalski2019introduction}. Such families are fundamental tools in both theoretical computer science (derandomization, coding theory, probabilistically checkable proofs, network design) and mathematics (mixing of random walks, growth in groups, sieve methods, topological and geometric applications). For a comprehensive overview, see \cite{hoory2006expander}.

The first explicit constructions of expander families were given by Margulis \cite{gregory1973margulis}, using arithmetic lattices and Kazhdan's property~$(T)$ from representation theory. Subsequently, explicit constructions achieving the largest possible spectral gap were discovered. The Alon--Boppana bound \cite[Theorem 5.3]{hoory2006expander} asserts that in any family of $d$-regular graphs one has
\[
\norm{\mathcal A |_{\1^\perp}} \geq 2\sqrt{d-1} - o(1)\qquad (N\to\infty),
\]
which provides a universal constraint on the spectral gap. An expander family is called \emph{Ramanujan} if for each graph in the family, we have $\norm{\mathcal A |_{\1^\perp}} \le 2\sqrt{d-1}$. Ramanujan families are thus optimal expanders, and the first explicit constructions were obtained by Lubotzky, Phillips, and Sarnak \cite{lubotzky1988ramanujan} and Margulis \cite{margulis1988explicit}, building on deep input from number theory.

\subsubsection*{Expansion in random graphs}
A complementary viewpoint, which is central to this paper, is that one can construct (non-explicit) expander families using randomness. In parallel with Margulis's work, Pinsker \cite{pinsker1973complexity} and independently Barzdin and Kolmogorov \cite{barzdin1993realization} showed that random $d$-regular graphs on $N$ vertices form expander families with high probability, that is $1 - o(1)$, as $N \to \infty$. Later, a far stronger result by Friedman \cite{friedman2003proof} established that random $d$-regular graphs are \emph{almost Ramanujan}, meaning that $\norm{\mathcal A |_{\1^\perp}} \leq 2\sqrt{d-1} + o(1)$ with high probability as $N \to \infty$.

One convenient way to generate random regular graphs is via permutations in the symmetric group $\Sym(N)$. Fix a symmetric generating set $S$ of size $2d$ (e.g.\ $S=\{\sigma_1^{\pm1},\dots,\sigma_{d}^{\pm1}\}$), where the $\sigma_i$ are chosen independently uniformly at random in $\Sym(N)$. The associated \emph{Schreier graph} on vertex set $[N] = \{1,\dots,N\}$ has edges
\[
\{j,\, \sigma_i(j)\}
\qquad 
(1 \leq i \leq d,\ j\in [N]),
\]
and is $2d$-regular (allowing parallel edges and loops). For asymptotic purposes, this model is equivalent to the uniform model of random $2d$-regular graphs \cite{wormwald1999models,janson2011random}. Friedman's proof is performed in this setting and relies on a delicate analysis of traces of high powers of the corresponding adjacency operator $\mathcal{A}$.

More generally, given any finitely generated group $\Gamma=\langle x_1,\dots,x_d\rangle$ and a sequence of finite permutation representations $\sigma_N \colon G_N \to \Sym(N)$, one can form random Schreier graphs as follows. Let $\phi_N\colon \Gamma\to G_N$ be a uniformly random homomorphism, set $\pi_N=\sigma_N\circ \phi_N\colon \Gamma\to \Sym(N)$, and build the Schreier graph on vertex set $[N]$ with edges
\[
    \{ j, \pi_N(x_i)(j) \} 
    \qquad 
    (1\le i\le d,\ j\in [N]).
\]
The random model above used in Friedman's proof is a special case of this, corresponding to $\Gamma=\F_d$ (the free group on $d$ generators) and $G_N=\Sym(N)$ with its natural action on $[N]$. Keeping $\Gamma=\F_d$ but allowing $G_N$ to range beyond $\Sym(N)$, variants of Pinsker's theorem (though not Friedman's almost-Ramanujan conclusion) have been proved for some important sequences of permutation groups. Most notably, expansion occurs with high probability for families of finite simple groups of fixed Lie type acting on themselves by left multiplication by the works of Bourgain and Gamburd \cite{bourgain2008uniform} and Breuillard, Green, Guralnick, and Tao \cite{breuillard2015expansion}.

\subsection{The polynomial method}


Recently, a new approach to operator-norm bounds in random permutation models as described above has emerged, known as the polynomial method by Chen, Garza-Vargas, Tropp, and van Handel \cite{chen2024new}. A thorough exposition of the method is available in \cite{van2025strong}. To put it simply, the method bounds the operator norm of random operators (such as adjacency operators of random Schreier graphs) by applying smooth functional calculus to the operator under the following prerequisites.

\begin{enumerate}[label={\Alph*.}]
    \item \emph{Asymptotic expansion of expected traces:} 
    For any $x \in \C[\Gamma]$ in the group algebra of $\Gamma$, obtain a rational expression for $\E[\tr(\pi_N(x))]$ in terms of $N$ with control over the degrees of the numerator and denominator. This leads to an asymptotic expansion of the expected trace as a Laurent series in $N$.

    \item \emph{Temperedness of leading coefficients:} Let $\lambda \colon \Gamma \to \bounded(\ell^2(\Gamma))$ be the left regular representation of $\Gamma$, given by $\lambda(\gamma)f(x)=f(\gamma^{-1}x)$ for $\gamma,x \in \Gamma$, $f \in \ell^2(\Gamma)$. Bound the growth of the first two leading coefficients in the Laurent series above along powers $x^p$ by at most the growth of $\norm{\lambda(x)}^p$ as $p \to \infty$.
    
    \item \emph{Unique trace property:} The reduced $C^*$-algebra $C_{\mathrm{red}}^\ast(\Gamma)$ is the norm-closure of $\lambda(\Gamma)$. This algebra is equipped with a canonical trace $\tau(a) = \langle a \delta_e, \delta_e \rangle$. This should be the only positive unital tracial functional on $C_{\mathrm{red}}^\ast(\Gamma)$.
\end{enumerate}
Assuming these conditions, the polynomial method yields \emph{strong convergence} of the random representations $\pi_N \colon \Gamma \to \unitary(N)$ to the left regular representation, meaning that for every $x \in \C[\Gamma]$, we have
\[
    \norm{\pi_N(x)} \to \norm{\lambda(x)}
    \quad
    \text{in probability as $N \to \infty$.}
\]

As shown in \cite{chen2024new}, this paradigm can be applied to the setting $\Gamma = \F_d$ and $G_N = \Sym(N)$ with its standard irreducible representation $\std \colon \Sym(N)\to \unitary(N-1)$, obtained by restricting the permutation action on $\C^N$ to the hyperplane $\1^\perp$. The result is strong convergence of the random representations $\pi_N = \std \circ \phi_N$ towards the left regular representation of $\F_d$. For the particular choice $x = \sum_{i=1}^d (x_i + x_i^{-1})$, this implies that the adjacency operator $\mathcal A$ of the random Schreier graph of $\Sym(N)$ acting via $\phi_N$ on $[N]$ satisfies $\norm{\mathcal A |_{\1^\perp}} \to 2 \sqrt{2d-1}$ (by Kesten \cite{kesten1958symmetric}) with high probability as $N \to \infty$. One thus recovers Friedman's theorem that random $2d$-regular graphs are almost Ramanujan.

\subsubsection*{Extensions and limitations of the polynomial method}
Since \cite{chen2024new}, the polynomial method has been extended in various directions. Magee and de la Salle \cite{MageeSalle2024} analyzed the case of $\unitary(N)$ instead of $\Sym(N)$ and demonstrated that strong convergence holds for numerous irreducible representations of $\unitary(N)$. Building on this, Cassidy \cite{cassidy2024random} explored additional irreducible representations of $\Sym(N)$ beyond $\std$ and proved strong convergence in those cases. Moreover, quantitative bounds on the convergence rate produced by the method were developed. This was first done by Chen, Garza-Vargas, and van Handel \cite{chen2024new2}, who specialized the machinery to classical ensembles and proved it can be used to improve the rate of strong convergence to free semicircular distributions in that setting. This was later strengthened by Magee, Puder, and van Handel \cite{magee2025strong} and Hide, Macera, and Thomas \cite{hide2025spectral}. The paper \cite{magee2025strong} substantially broadened the framework of the original polynomial method by replacing the requirement for exact rationality of $\E \left[ \tr \pi_N(x) \right]$ with sufficiently analytic-like asymptotic expansion as a Laurent series in $N$. They apply this to prove strong convergence when $\Gamma$ is the fundamental group of a closed orientable surface of genus at least $2$, with $G_N = \Sym(N)$ and the $\std$ representation. 

However, not all natural random permutation models fit within even the most extended known frameworks. As the following example demonstrates, the presence of nontrivial \emph{torsion} in $\Gamma$ introduces new phenomena in the asymptotic expansions of $\E[\tr(\pi_N(\gamma))]$. Specifically, the asymptotic expansion is no longer a simple Laurent series in $N$ but instead includes fractional powers of $N$.

\begin{example*}
Let $\Gamma = C_2 = \langle \gamma \rangle$, the cyclic group of order $2$. A random homomorphism $\phi_N\colon C_2\to \Sym(N)$ corresponds to choosing a random involution $\sigma \in \Sym(N)$. Then, since $\pi_N = \std \circ \phi_N$,
\[
    \E_{\pi_N} \left[ \tr(\pi_N(\gamma)) \right]
    = \E_{\sigma} \left[ |\fix(\sigma)| \right] - 1.
\]
Let $a_N$ be the number of involutions in $\Sym(N)$. Then, by symmetry,
\[
    \E_{\sigma} \left[ |\fix(\sigma)| \right]
    = N \P(\sigma(1)=1)
    = N a_{N-1} / a_N.
\]
The involution numbers satisfy $a_N = a_{N-1}+(N-1)a_{N-2}$, hence the ratio $r_N = a_N/a_{N-1}$ satisfies $r_N = 1 + (N-1)/r_{N-1}$. From this recursion we obtain that $r_N = \sqrt{N} + O(1)$ as $N \to \infty$. Therefore
\[
    \E_{\pi_N} \left[ \tr(\pi_N(\gamma)) \right]
    = \sqrt{N} + O(1).
\]
\end{example*}

Related torsion-constrained permutation models were previously analyzed by Bordenave and Collins~\cite{bordenave2019eigenvalues}, and by Shen and Wu~\cite{shen2025belyi}. Both works consider models where the generators of $\Gamma$ are mapped to permutations of finite orders, but impose the restriction that these permutations be \emph{fixed-point-free}. 

Bordenave and Collins~\cite{bordenave2019eigenvalues} develop a refined trace method based on non-backtracking operators, which predates the polynomial method. While their method provides yet another alternative proof of Friedman's theorem, it also yields operator-norm limits for random graphs generated by fixed-point-free involutions (equivalently, unions of random perfect matchings). Their fixed-point-free setting is tailored to the permutation models arising from random lifts, and their methods do not readily extend to uniform sampling of involutions in $\Sym(N)$. On the other hand, Shen and Wu~\cite{shen2025belyi} analyzed spectral gaps in random Belyi surfaces constructed from the Brooks--Makover model, which glues hyperbolic triangles according to two random permutations of orders $2$ and $3$, both again restricted to be fixed-point-free in order for the gluing to make sense. In this setup, expected traces turn out to exhibit the same asymptotic expansions with integer powers of $N$ as those in surface groups (see \cite[Proposition 1.2]{shen2025belyi}), thus fitting within the existing frameworks of the extended polynomial method of Magee--Puder--van Handel. 

The fixed-point-free condition is restrictive: for a uniformly random element of order dividing $d$ in $\Sym(N)$, the probability of having no fixed points decays like $\exp(-N^{1/d})$ as $N \to \infty$.\footnote{The counting generating functions are $\exp(P(z))$ for all solutions of $x^d=1$ and $\exp(P(z)-z)$ for fixed-point-free ones, where $P(z)=\smash[b]{\sum_{k\mid d}} z^k/k$. By \cite[Theorem 1]{muller1997finite}, their coefficient ratio is asymptotically $\exp(-N^{1/d})$ up to a constant factor.} In addition, forbidding fixed points removes the torsion-driven leading terms in expected traces in the uniform model. The $\sqrt{N}$ term in the uniform $C_2$ example above is exactly such a fixed-point contribution.

\subsection{Expansion in random graphs with torsion restrictions}

In this paper, we show how to handle general torsion constraints in random permutation models. We impose the relations independently on the generators by taking
\begin{equation}
    \label{eq:Gamma_free_product_of_Gi}
    \Gamma = G_1 * G_2 * \cdots * G_m,
\end{equation}
where each $G_i$ is a finite group and $m \geq 2$. As before, let $\phi_N \in \hom(\Gamma, \Sym(N))$ be chosen uniformly at random, and define $\pi_N = \std \circ \phi_N \colon \Gamma \to \unitary(N-1)$. We are able to push the polynomial method to cover the new phenomena arising from torsion in trace asymptotics and, importantly, verify the resulting hypotheses for free products of finite groups. This yields strong convergence of $\pi_N$ to the left regular representation of $\Gamma$ in our model.

\begin{theorem}[strong convergence for free products of finite groups]
    \label{strong_convergence}
Let $\Gamma$ be a free product of finite groups as in \cref{eq:Gamma_free_product_of_Gi}. Then the random representations $\pi_N = \std \circ \phi_N$ strongly converge to the left regular representation of $\Gamma$. 
\end{theorem}

Before discussing the methods in more detail, we highlight two special cases covered by Theorem~\ref{strong_convergence} where the limiting spectral radius can be computed explicitly, thus yielding asymptotically explicit spectral gaps for the random Schreier graphs.

\begin{example*}
Consider $\Gamma = C_2 * C_3 \cong \PSL_2(\Z)$, the modular group, with generators $x$ and $y$ of orders $2$ and $3$, respectively. Random homomorphisms $\phi_N$ correspond to choosing a random involution and a random element of order $3$ in $\Sym(N)$ (compare with \cite{shen2025belyi}, where the permutations are required to be fixed-point-free). The associated Schreier graphs are random $3$-regular graphs. By Theorem~\ref{strong_convergence}, their second-largest eigenvalues converge with high probability to $\norm{\lambda(x + y + y^2)}$. Adjacency operators in free products of two finite groups are well understood by the work of Cartwright and Soardi \cite{cartwright1986random}. It is possible to determine their spectral measure explicitly (see the curious Figure \ref{fig:spectral-measure_C2C3}), and in particular to compute the spectral radius. In the present case, it is given by
\[
    \frac{1 + \sqrt{13 + 8 \sqrt{2}}}{2}
    \approx 2.9654.
\]
This value exceeds the Ramanujan bound for $3$-regular graphs, which is $2\sqrt{2} \approx 2.8284$. Thus, we obtain a simple random model for generating non-Ramanujan $3$-regular graphs with an explicit spectral gap.

\begin{figure}[t]
\includegraphics[width=0.6\textwidth]{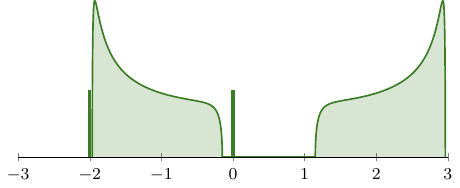}
\caption{The spectral measure of the adjacency operator of $C_2*C_3$. The support of the measure consists of the union of two intervals with boundary points at $(1 \pm \sqrt{13 + 8 \sqrt{2}})/2$ and $(1 \pm \sqrt{13 - 8 \sqrt{2}})/2$, together with discrete atoms at $-2$ and $0$, each with mass $1/6$. These atoms correspond to torsion-induced $\ell^2$-eigenfunctions $f$ of the adjacency operator $\mathcal A f(g)=f(xg)+f(yg)+f(y^2g)$. Imposing that $f$ has sum zero on each $C_3$-triangle $\{g,yg,y^2g\}$ reduces the eigenvalue equation $\mathcal A f = \lambda f$ to $f(xg)=(\lambda+1)f(g)$, and the relation $x^2=1$ forces $\lambda+1 = \pm1$.}
\label{fig:spectral-measure_C2C3}
\end{figure}
\end{example*}

\begin{example*}
Consider $\Gamma = C_2 * C_2 * \cdots * C_2$ with $d \geq 3$ copies of $C_2$. Let $x_1, x_2, \dots, x_d$ denote the corresponding generators of order $2$. The Cayley graph of $\Gamma$ with respect to these generators is the $d$-regular tree, whose adjacency operator has spectral radius $2\sqrt{d-1}$. By Theorem~\ref{strong_convergence}, random Schreier graphs generated by $d$ uniformly random involutions in $\Sym(N)$ (allowing fixed points) are almost Ramanujan with high probability as $N \to \infty$. These graphs can be interpreted as unions of $d$ random matchings, which need not be perfect (compare with \cite{bordenave2019eigenvalues}). More generally, one can analyze settings where the adjacency operator on $\Gamma$ is weighted by assigning different weights to the generators $x_i$. In such cases, the spectral radius can be computed in closed form in terms of these weights and $d$ (see \cite[Example~1]{AomotoKato1988}).

The case when $\Gamma = C_2 * C_2$ is rather different. This group is amenable, implying that its adjacency operator has spectral radius $2$. The corresponding random Schreier graphs are generated from two random involutions, so they are unions of cycles on $N$ vertices. With high probability as $N \to \infty$, we get a disconnected graph, hence $\norm{\mathcal A |_{\1^\perp}} = 2$. Strong convergence here comes for free.
\end{example*}

Except for $C_2 * C_2$, none of the groups in \cref{eq:Gamma_free_product_of_Gi} is amenable. By selecting a symmetric generating set $S$, the corresponding adjacency operator $\sum_{s \in S} s$ has a norm strictly less than $|S|$. As a direct consequence of strong convergence, this implies that, with high probability as $N \to \infty$, the random Schreier graphs in this model are connected. Equivalently, a random homomorphism $\phi_N \colon \Gamma \to \Sym(N)$ has a transitive image with high probability. This phenomenon is well-known for groups as in \cref{eq:Gamma_free_product_of_Gi} and even for broader families, as proved by Liebeck and Shalev \cite[Theorem 1.12]{liebeck2004fuchsian}. Their methods for proving transitivity rely on character bounds for $\Sym(N)$ and an analysis of the associated representation zeta function. 
Strong convergence provides a complementary spectral perspective on these results.\footnote{Using the classification of finite simple groups, Liebeck and Shalev further strengthen this, showing that $\langle \image \phi_N \rangle \geq \Alt(N)$ with high probability as $N \to \infty$. The specific models considered in our examples were treated earlier in \cite{liebeck1996classical}. Circumventing the use of classification via spectral methods would require dealing with all irreducible representations of $\Sym(N)$ instead of just $\std$.}

\subsection{Methods}

We now explain how the polynomial method can be pushed further to cover asymptotic expansions with fractional powers of $N$, and then how we verify the resulting requirements.

\subsubsection*{Polynomial method beyond integer exponents}

We follow the framework of Magee, Puder, and van Handel \cite{magee2025strong}. Let $\Gamma$ be a finitely generated group with a symmetric generating set $S$. For any $\gamma \in \Gamma$, let $\abs{\gamma}$ denote the word length of $\gamma$ with respect to $S$. For each $N$, consider a random representation $\pi_N = \std \circ \phi_N$ as above, and let $\ntr(\pi_N(\gamma)) = \tr(\pi_N(\gamma)) / N$ denote the normalized trace.

We first discuss the requirement for the asymptotic expansion of $\E[\ntr(\pi_N(\gamma))]$. In \cite{magee2025strong}, the authors identify an analytic-like property of this expansion: for any specified precision $s \geq \abs{\gamma}$, a Laurent series truncated at $s-1$ terms approximates the expected trace with an error of $O(s)^{O(s)} N^{-s}$. In our setting, we require similar control, but with $N$ replaced by $N^{1/\mu}$, where $\mu$ is a positive integer depending on $\Gamma$. Here is the precise formulation.

\begin{requirement} \label{requirement:asymptotic_expansion}
There exists a constant $C\ge 1$ such that for every $\gamma\in \Gamma$, every $s \ge |\gamma|$, and every $N \ge Cs^C$, there exist constants $u_k(\gamma)\in \R$ for all non-negative integers $k$ such that $u_0(\gamma)=\tau(\lambda(\gamma))$ and
\[
\left| \E\!\left[ \ntr \pi_N(\gamma) \right] - \sum_{k=0}^{s-1} \frac{u_k(\gamma)}{N^{k/\mu}} \right|
\le \frac{(Cs)^{Cs}}{N^{s/\mu}} .
\]
\end{requirement}

In the second ingredient of the polynomial method, it is crucial to control the values of $u_k$. Note that each $u_k(\gamma)$ is uniquely determined by the approximation above. Moreover, since the trace can be expressed as the number of fixed points, we have $u_k(\gamma)=u_k(\gamma^{-1})$, and by conjugation invariance of the trace, the $u_k$ are class functions. We extend the $u_k$ linearly to $\C[\Gamma]$. Following \cite{MageeSalle2024}, we say a functional $u\colon \C[\Gamma]\to \C$ is \emph{tempered at $x\in \C[\Gamma]$} if
\[
\limsup_{p\to\infty} |u(x^p)|^{1/p}\le \|\lambda(x)\|.
\]
The polynomial method of \cite{chen2024new} and \cite{magee2025strong} requires temperedness of the functional $u_1$. In our setting, we require temperedness of all the first $\mu$ functionals.

\begin{requirement} \label{requirement:temperedness}
For every positive integer $k \leq \mu$, the functional $u_k$ is tempered at all self-adjoint elements of $\C[\Gamma]$. 
\end{requirement}

Under these requirements, we can extend the polynomial method following the approach of \cite{magee2025strong} to obtain \emph{upper} bounds on the operator norm of $\pi_N(x)$ for every $x \in \C[\Gamma]$.

\begin{theorem} \label{thm:main}
Suppose that \Cref{requirement:asymptotic_expansion} and \ref{requirement:temperedness} hold.
Then for every $x\in\C[\Gamma]$,
\[
    \norm{\pi_N(x)} \leq \norm{\lambda(x)} + o(1)
    \quad
    \text{with high probability as $N \to \infty$}.
\]
\end{theorem}

The proof in fact yields a quantitative form of this bound.\footnote{For each fixed $x\in\C[\Gamma]$ and $\varepsilon>0$, the probability for the event $\norm{\pi_N(x)} \geq \norm{\lambda(x)} + \varepsilon$ decays like $\epsilon^{-O(1)} N^{-1/\mu}$ as $N\to\infty$, where $\mu$ is the integer from \Cref{requirement:asymptotic_expansion}. See Remark~\ref{remark:rate_of_convergence}.} To achieve strong convergence, it remains to establish an analogous \emph{lower} bound. However, under the unique trace property (requirement C of the original polynomial method), this lower bound follows automatically from the upper bound via a general $C^*$-algebraic argument, as explained in \cite[Section~2.3]{van2025strong}. Fortunately, this property holds for many groups. By \cite{Breuillard2017}, the unique trace property is equivalent to the group having trivial amenable radical, which includes every free product of finite groups as in \cref{eq:Gamma_free_product_of_Gi} apart from the exceptional case $C_2 * C_2$ (which was addressed in the example above). Thus, strong convergence in \Cref{strong_convergence} follows once we verify the two requirements above in our setting.

\subsubsection*{Verifying \Cref{requirement:asymptotic_expansion}}

Let $\Gamma$ be a free product of finite groups as in \cref{eq:Gamma_free_product_of_Gi}, and let $\gamma \in \Gamma$. The first step in obtaining a suitable asymptotic expansion for $\E[\tr \pi_N(\gamma)]$ is to consider the local case when $\gamma$ belongs to one of the free factors $G_i$. In this scenario, the expected trace depends solely on the distribution of the image of $G_i$ under a uniformly random homomorphism $G_i \to \Sym(N)$. For instance, if $G_i = C_2$ with generator $\gamma$, then $\E[\tr \pi_N(\gamma)] = N a_{N-1} / a_N$, as shown in the previous example. Asymptotic expansions for the involution numbers were obtained by Moser and Wyman \cite{moser1955solutions}. These asymptotics arise from analyzing the generating function $\exp(z + z^2/2)$ at its dominant singularity. This was pushed much further by Müller \cite{muller1997finite}, building on more general work by Harris and Schoenfeld \cite{harris1968asymptotic}, who provided an asymptotic expansion for any element $\gamma$ in any finite group, in particular in any free factor $G_i$. The result is that the expected trace $\E[\tr(\pi_N(\gamma))]$ admits an expansion in powers of $N^{-1/|G_i|}$. However, these expansions include error terms of the form $O_s(N^{-s/|G_i|})$, which are not sufficiently strong for our purposes. In \Cref{requirement:asymptotic_expansion}, we need the implicit constant in the error to depend on $s$ at most doubly exponentially. Our first job is thus to make the analysis of both \cite{harris1968asymptotic} and \cite{muller1997finite} more quantitative by carefully tracking the dependence of the error terms on $s$. This is the technical crux of the paper.

After dealing with the local scenario when $\gamma$ belongs to a single free factor, we must handle the general case. For this, we rely on the work of Puder and Zimhoni \cite{PuderZimhoni2024}, who reformulate $\E[\tr \pi_N(\gamma)]$ in terms of covering spaces of a CW complex $X_\Gamma$ with $\pi_1(X_\Gamma)\cong \Gamma$. Homomorphisms $\Gamma \to \Sym(N)$ are in bijective correspondence with $N$-sheeted covering spaces of $X_\Gamma$ (see \cite{hatcher2002algebraic}). The trace $\tr \phi_N(\gamma)$, being the number of fixed points of $\phi_N(\gamma)$, corresponds to the number of lifts of the cellular map $Y_\gamma \to X_\Gamma$ into the covering space corresponding to $\phi_N$. In \cite{PuderZimhoni2024}, the authors develop a systematic approach to analyze these lifts via a combinatorial decomposition of the map $Y_\gamma \to X_\Gamma$ into simpler constituent maps, which are then analyzed using the local asymptotics described above. As before, we must make this analysis quantitative to ensure all dependencies are at most doubly exponential in $s$. The key observation is that the number of surjective cellular maps from the complex $Y_\gamma$ with $O(s)$ vertices and edges is at most $O(s)^{O(s)}$. By synthesizing these results, the constant $C$ in \Cref{requirement:asymptotic_expansion} depends only on the sizes of the groups $G_i$, and the integer $\mu$ is the least common multiple of their orders.

\subsubsection*{Verifying \Cref{requirement:temperedness}}

Let $x \in \C[\Gamma]$ be self-adjoint. We need to control the growth of $u_k(x^p)$ as $p \to \infty$ for each $1 \leq k \leq \mu$. The key tool here is the positivization trick as developed by Magee and de la Salle \cite{MageeSalle2024}, which allows us to reduce the problem to elements $x$ that are symmetric generators of random walks on $\Gamma$. This means that in the expansion of $x$ in the group basis, the coefficients are all nonnegative, sum to $1$, and are stable under inversion. For such $x$, the powers $x^p$ can be expressed in terms of hitting probabilities $\P(X_1 X_2 \dots X_p = \gamma)$ of the associated random walk on $\Gamma$ after $p$ steps. At this point, specific group-theoretic properties of $\Gamma$ come into play. The hitting probabilities depend strongly on whether $\gamma$ is a torsion element or not, and in the latter case, whether $\gamma$ is a proper power or contained in a subgroup isomorphic to $C_2 * C_2$. We address each scenario separately, using explicit expressions for $u_0(\gamma)$ derived from \cite{PuderZimhoni2024}. We then apply a version of Friedman's spectral trick from \cite{van2025strong,magee2025strong} to sum over all $\gamma$, making it possible to control the hitting probability of landing in a specific scenario among those described above. Ultimately, this controls the growth of $u_k(x^p)$.

\begin{remark*}
The route to \Cref{requirement:asymptotic_expansion} is in many ways specific to the symmetric group, since the fixed-point statistics of words can be analyzed very explicitly and combinatorially using the machinery of covering spaces. It would be interesting to develop analogues for other families of finite groups and their natural permutation actions, such as $\PSL_2(\F_p)$ acting on the projective line $\P^1(\F_p)$. The main challenge is that one must restrict to covering spaces whose monodromy groups are subgroups of $\PSL_2(\F_p)$. Nevertheless, computational experiments suggest that the spectral norm $\norm{\mathcal A_{\1^\perp}}$ in Schreier graphs arising from $\Gamma$ and random homomorphisms into $\PSL_2(\F_p)$ instead of the symmetric group approaches the corresponding spectral radius of $\Gamma$ for both the free group $\F_2$ and the free product $C_2 * C_3$.
\end{remark*}

\subsection{Reader's guide}

Section~\ref{sec:assumptions} shows how Requirements~\ref{requirement:asymptotic_expansion} and~\ref{requirement:temperedness} together imply the strong convergence upper bound of Theorem~\ref{thm:main}. Section~\ref{sec:asymptotics} develops quantitative asymptotics for Taylor coefficients of exponentials of polynomials. Section~\ref{sec:fractional_expansions} then turns these asymptotics into fractional expansions for normalized counts of group actions. Section~\ref{sec:expected_traces} translates those expansions, using the theory of covering spaces, into the expansion needed in \Cref{requirement:asymptotic_expansion}. Finally, Section~\ref{sec:temperedness} establishes \Cref{requirement:temperedness} by analyzing hitting probabilities of random walks on $\Gamma$, thereby completing the proof of Theorem~\ref{strong_convergence}.

\section{Assumptions imply strong convergence}
\label{sec:assumptions}

The basic principle used in the polynomial method is an inequality due to Markov (see \cite[Lemma 3.6]{van2025strong}). We will use the following more elaborate version of this inequality.

\begin{lemma}[Lemma 4.2 in \cite{MageeSalle2024}]
\label{lem:markov-brothers-interpolated}
For every real polynomial $P$ of degree at most $q$ and every positive integer $k$, we have
\[
\sup_{t \in [ 0, 1/(2q^{2}) ]}
\abs{P^{(k)}(t)}
\leq
\frac{2^{2k+1}q^{4k}}{(2k-1)!!}
\sup_{n \geq q^{2}}
\abs{P(1/n) }.
\]
\end{lemma}

Our next goal is extending the approximation of $\E[\ntr \pi_N(\gamma)]$ as given in \Cref{requirement:asymptotic_expansion} to an analogous approximation for $\E[\ntr h(\pi_N(x))]$ for $x \in \C[\Gamma]$. In the polynomial method, this is called the \emph{master inequality} \cite[Corollary 3.9]{van2025strong}. Our proof here mimics that of \cite[Lemma 3.1]{magee2025strong}. The points of divergence are the polynomial $f_h$ in \cref{eq:f_h}, its evaluation at $N^{-1/\mu}$ in \cref{eq:f_h_bound}, and the fact that we need to apply Lemma \ref{lem:markov-brothers-interpolated} to a considerably higher order of derivative in \cref{eq:derivativeBound}.

Once and for all, fix a symmetric generating set $S$ of $\Gamma$. For $x \in \C[\Gamma]$, let $\abs{x}$ be the maximum word length of a group element in the support of $x$, and let $\norm{x}_{C_{\mathrm{red}}^\ast(\Gamma)}$ be the supremum of $\norm{\pi(x)}$ over all unitary representations $\pi$ of $\Gamma$. For a polynomial $h$, we denote $\norm{h}_I = \sup_{t \in I} |h(t)|$.

\begin{proposition}[master inequality]
\label{master_inequality} 
Suppose \Cref{requirement:asymptotic_expansion} holds with constant $C$. Let $x \in \C[\Gamma]$ be self-adjoint with $K = \norm{x}_{C_{\mathrm{red}}^\ast(\Gamma)}$. Then for every positive integer $N$, every real polynomial $h$ of degree at most $q$, and every positive integer $k \leq \mu$, we have
\[ 
\abs{u_{k}(h(x))} \leq
O_{\mu, |S|, C, |x|} \left( q^{O_{\mu, C}(1)} \norm{h}_{[-K,K]} \right),
\]
and
\[ 
\abs{ 
    \E \left[ 
    {\ntr h(\pi_{N}(x))} \right] 
    - 
    \sum_{k=0}^{\mu} \frac{u_k(h(x))}{N^{k/\mu}} 
} 
\leq
O_{\mu, |S|, C, |x|} \left( \frac{q^{O_{\mu, C}(1)} \norm{h}_{[-K,K]}}{N^{1 + 1/\mu}} \right).
\] 
\end{proposition}
\begin{proof}
In case $x = 0$ or $h$ is of degree $0$, we have $u_k(h(x)) = 0$ and both inequalities hold trivially. Assume therefore that $\abs{x} \geq 1$ and $q \geq 1$. Write
\[ 
h(x) = 
\sum_{\gamma \in \Gamma}
\alpha_{\gamma} \gamma 
\quad (\alpha_{\gamma} \in \C).
\]
The maximum length of a monomial in $h(x)$ does not exceed $q \abs{x}$. Thus the number of nonzero coefficients $\alpha_\gamma$ is at most $(2\abs{S}+1)^{q|x|} \le (3\abs{S})^{q|x|}$. Apply \Cref{requirement:asymptotic_expansion} separately to each element of the support of $h(x)$ with order of approximation $q \abs{x} + \mu + 1 \leq 3 \mu q \abs{x}$. As long as $N \geq C (3 \mu q \abs{x})^C$, we thus obtain
\begin{equation}
    \label{eq:Etrace_first_estimate}
\abs{
    \E \left[ {\ntr h(\pi_{N}(x))} \right] 
    - 
    f_h ( N^{-1/\mu} )
}
\leq
\frac{(3\mu Cq|x|)^{3\mu Cq|x|}}{N^{(q|x|+\mu+1)/\mu}}
\sum_{\gamma\in\Gamma} \abs{\alpha_{\gamma}},
\end{equation}
where we have denoted
\begin{equation}
    \label{eq:f_h}
    f_{h}(t) = 
    \sum_{k = 0}^{q|x|+\mu} u_k(h(x)) t^{k} =
    \sum_{\gamma\in\Gamma}\sum_{k=0}^{q|x|+\mu} \alpha_\gamma u_{k}(\gamma) t^{k}.
\end{equation}
Note that since $h(x)$ is self-adjoint and $u_k(\gamma) = u_k(\gamma^{-1})$, the coefficients $u_{k}(h(x))$ are real numbers. Thus $f_{h}$ is a real polynomial. The sum of the coefficients $\alpha_\gamma$ in \cref{eq:Etrace_first_estimate} runs over a set of at most $(3\abs{S})^{q|x|}$ elements. By Cauchy-Schwarz we thus have
\[
        (3\abs{S})^{-q|x|/2} \sum_{\gamma \in \Gamma} \abs{\alpha_{\gamma}}
        \leq
        \norm{h(x)}_{\ell^2(\Gamma)}
        =
        \norm{h(\lambda(x)) \delta_e}_{\ell^2(\Gamma)}
        \leq
        \norm{h}_{[-K,K]},
\]
where we have used functional calculus and the assumption $\norm{x}_{C_{\mathrm{red}}^\ast(\Gamma)} \leq K$ in the last inequality. It now follows that, as long as $N \geq (3\abs{S})^{1/2}(3 \mu C q \abs{x})^{3 \mu C}$, we have
\begin{equation}
\label{eq:exp}
\abs{
    \E \left[ {\ntr h(\pi_{N}(x))} \right] 
    - 
    f_h ( N^{-1/\mu} )
}
\leq
\frac{\norm{h}_{[-K,K]}}{N^{1 + 1/\mu}}.
\end{equation}
Since $\norm{\pi_N(x)} \leq \norm{x}_{C_{\mathrm{red}}^\ast(\Gamma)} = K$, the expectation $\E \left[ \ntr h(\pi_N(x)) \right]$ is at most $K/N$ in absolute value. Thus by functional calculus and \cref{eq:exp}, we obtain 
\begin{equation}
    \label{eq:f_h_bound}
\abs{f_{h} (N^{-1/\mu})}
\leq
2 \norm{h}_{[-K,K]}.
\end{equation}
Now, we are in a position to estimate the derivatives of $f_h$. Apply Lemma \ref{lem:markov-brothers-interpolated} to the real polynomial $f_h$ of degree at most $(3\abs{S})^{1/4}(3\mu Cq \abs{x})^{3 \mu C/2}$ to estimate
\begin{equation}
    \label{eq:derivativeBound}
    \norm{f_{h}^{(k)}}_{\left[ 0, 1/(2 (3r)^{1/2} (3\mu Cq |x|)^{3 \mu C}) \right]}
    \leq
    \frac{4^{k+1} \cdot (3\abs{S})^k (3\mu Cq |x|)^{6 \mu C k}}{(2k-1)!!} \norm{h}_{[-K,K]}.
\end{equation}
For every integer $1 \leq k \leq \mu$, we have $u_k(h(x)) = f_h^{(k)}(0)/k!$. The first part of the statement then follows immediately from \cref{eq:derivativeBound}.

Using the same bounds, we can also deduce the second part of the statement. First of all, in the regime when $N \geq (2 (3\abs{S})^{1/2} (3 \mu C q \abs{x})^{3 \mu C})^\mu$, Taylor expanding $f_h$ at $0$ up to degree $\mu$ gives
\[
\begin{split}
    \abs{
        f_{h} (N^{-1/\mu}) - \sum_{k=0}^{\mu} \frac{u_k(h(x))}{N^{k/\mu}}
    } 
    & \leq 
    \frac{1}{(\mu + 1)! N^{1 + \frac{1}{\mu}}} \norm{ f_{h}^{(\mu+1)} }_{[ 0,1/N^{-1/\mu} ]} \\
    & \leq
    O_{\mu, |S|, C, |x|} \left( \frac{q^{O_{\mu, C}(1)} \norm{h}_{[-K,K]}}{N^{1 + 1/\mu}} \right).
\end{split}
\]
The second part of the statement in this regime thus follows from \cref{eq:exp} by triangle inequality.
On the other hand, in the regime when $N < (2 (3\abs{S})^{1/2} (3 \mu C q \abs{x})^{3 \mu C})^\mu = O_{\mu, |S|, C,|x|}(q^{O_{\mu, C}(1)})$, we can directly estimate, using $\abs{u_0(h(x))} = \abs{\tau(h(\lambda(x)))} \leq \norm{h}_{[-K,K]}$, 
\[
\begin{split}
\abs{\E \left[{\ntr h(\pi_{n}(x))} \right] - \sum_{k=0}^{\mu} \frac{u_k(h(x))}{N^{k/\mu}}}
& \leq 
2 \norm{h}_{[-K,K]} 
+ 
\sum_{k=1}^{\mu} \frac{\abs{u_k(h(x))}}{N^{k/\mu}} \\
& \leq 
O_{\mu, |S|, C, |x|} \left( \frac{q^{O_{\mu, C}(1)} \norm{h}_{[-K,K]}}{N^{1 + 1/\mu}} \right).
\end{split}
\]
This establishes the second part of the statement in this regime as well. The proof is complete.
\end{proof}

From here on, we follow the usual steps of the polynomial method. Let us first extend Proposition \ref{master_inequality} to all smooth functions $h$.

\begin{proposition}
\label{smooth_master}
Suppose \Cref{requirement:asymptotic_expansion} holds with constant $C$. Let $x \in \C[\Gamma]$ be self-adjoint with $K = \norm{x}_{C^{*}(\Gamma)}$. 
\begin{enumerate}
    \item For all positive integers $k \leq \mu$, the functional $\nu_k \colon h \mapsto u_k(h(x))$ extends from a polynomial domain to a compactly supported distribution on $C^\infty(\R)$.
    \item For all positive integers $N$ and all $h \in C^\infty(\R)$, we have
\[ 
\abs{
    \E \left[ {\ntr h(\pi_{N}(x))} \right] 
    - \tau (h(\lambda(x))) 
    - \sum_{k=1}^\mu \frac{\nu_k(h)}{N^{k/\mu}} 
}
\leq
O_{\mu,|S|,C,|x|} \left( \frac{\norm{f^{(\omega)}}_{[0,2\pi]} }{N^{1 + \frac{1}{\mu}}} \right), 
\]
where $\omega$ depends only on $\mu$ and $C$, and $f(\theta)=h(K\cos\theta)$. 
\end{enumerate}
\end{proposition}
\begin{proof}
The first portion of the statement about extending $\nu_k$ to a compactly supported distribution follows at once from the first part of Proposition \ref{master_inequality} and \cite[Lemma 4.7]{chen2024new}. For the second one, we assume first that $h$ is a polynomial of degree $q$. Expand $h$ in the basis of Chebyshev polynomials $T_\ell(\cos \theta ) = \cos(\ell \theta)$ as
\[
h(t) = \sum_{\ell=0}^{q} a_{\ell} T_{\ell}(t/K).
\]
Applying the second part of Proposition \ref{master_inequality} separately to each
Chebyshev polynomial yields 
\[ 
\abs{ 
    \E \left[ {\ntr h(\pi_{N}(x))} \right] 
    - \tau ( h(\lambda(x)) )
    - \sum_{k=1}^\mu \frac{u_k(h(x))}{N^{k/\mu}} 
}
\leq
\frac{O_{\mu,|S|,C,|x|}(1)}{N^{1 + 1/\mu}} 
\sum_{\ell = 1}^{q} \ell^{O_{\mu,C}(1)} |a_{\ell}|. 
\]
The second part of the statement for $h$ a polynomial now follows immediately from \cite[Lemma 2.3]{magee2025strong}. Note that the resulting bound does not depend
on the degree of $h$, hence, by density of polynomials in $C^{\infty}(\R)$,
the statement extends by continuity to every smooth $h$.
\end{proof}

We are now ready to conclude the proof of Theorem \ref{thm:main}.

\begin{proof}[Proof of Theorem \ref{thm:main}]
Suppose first that $x \in \C[\Gamma]$ is self-adjoint with $K = \norm{x}_{C^{*}(\Gamma)}$. Let $\omega$ and $\nu_k$ be as in Proposition \ref{smooth_master}. By \cite[Lemma 4.10]{chen2024new}, for any $\epsilon > 0$, there exists a smooth function $h \in C^{\infty}(\R)$ with the following properties:
\begin{itemize}
\item $h(z) \in [0,1]$ for all $z$,
\item $h(0) = 0$ for $|z| \leq \norm{\lambda(x)} + \epsilon/2$ and $h(z)=1$ for $|z| \geq \norm{\lambda(x)} + \epsilon$,
\item $\|f^{(\omega)}\|_{[0,2\pi]} \leq O_{\mu, C, K}(\epsilon^{-O_{\mu,C}(1)})$ for $f(\theta)=h(K\cos\theta)$.
\end{itemize}
Thus $\tau(h(\lambda(x))) = 0$. Moreover, by \Cref{requirement:temperedness} and \cite[Lemma 4.9]{chen2024new}, we have $\supp (\nu_k) \subseteq [- \norm{\lambda(x)}, \norm{\lambda(x)}]$, so that $\nu_k(h) = 0$ for all $1 \leq k \leq \mu$. Now,
\[ 
\P \left[ \norm{\pi_{N}(x)} \geq \norm{\lambda(x)} + \epsilon \right]
\leq
\P \left[ \tr h(\pi_{N}(x)) \geq 1 \right]
\leq
\E \left[ {\tr h(\pi_{N}(x))} \right],
\]
and the latter can be bounded from Proposition \ref{smooth_master} by
\[
    O_{\mu,|S|,C,|x|,K} \left( \frac{1}{\epsilon^{O_{\mu,C}(1)} N^{1/\mu}} \right).
\]
If we take $\epsilon = N^{-a}$ with $a$ small enough so that $1/a \gg_{\mu, C} 1$, the latter bound is $O_{\mu,|S|,C,|x|,K}(N^{-1/(2\mu)})$. This completes the proof in case when $x$ is self-adjoint. The general case then follows from the fact that for any $y \in \C[\Gamma]$, we have $\norm{y}^2 = \norm{y^\ast y}$, and running the argument above with $x = y^\ast y$.
\end{proof}

\begin{remark}
    \label{remark:rate_of_convergence}
It follows from the proof above that we can also control the rate of convergence in Theorem \ref{thm:main}. Specifically, if we choose $a$ sufficiently small in terms of $\mu$ and $C$, the proof shows that
\[
    \P \left[ \norm{\pi_{N}(x)} \geq \norm{\lambda(x)} + N^{-a/2} \right] \leq O_{\mu,|S|,C,|x|,K}(N^{-1/(2\mu)}).
\]
\end{remark}

\section{Asymptotics for exponentials of polynomials}
\label{sec:asymptotics}

Let $P(z) = \sum_{i = 1}^q c_i z^i$ be a polynomial of degree $q \geq 2$ with real coefficients $c_i$ satisfying
\begin{equation} \label{eq:polynomial_coefficient_conditions}
    c_i \geq 0 \quad \text{for } 1 \leq i \leq q,
    \qquad c_1 = 1,
    \qquad c_q = 1/q,
    \qquad c_i = 0 \quad \text{for } i \nmid q.
\end{equation}
Consider the entire function $\exp(P(z))$, and let its Taylor expansion at $0$ be
\[
    \exp(P(z)) =
    \sum_{n = 0}^\infty \alpha_n z^n.
\]
Müller \cite{muller1997finite} provides explicit formulas, in terms of the coefficients of $P(z)$, for the asymptotic expansion of the Taylor coefficients $\alpha_n$ as $n \to \infty$. We recall these formulas here, and make explicit all dependencies that we require for our application.

\subsection{Quantitative asymptotics in terms of $r_n$}

Müller's formulas build on more general results of Harris and Schoenfeld \cite{harris1968asymptotic} for generating functions
\[
    f(z)=\sum_{n=0}^\infty \alpha_n z^n
\]
that are analytic in a neighborhood of the origin and satisfy the conditions (A)--(E) listed below.
    
\begin{enumerate}[(A)]
    \item\label{cond:A} The function $f(z)$ is analytic for $\abs{z} < R$ for some $R \in (0, \infty]$ and is real for real $z$.
    \item There exists $R_0 \in (0,R)$ and a real-valued function $d(r)$ defined on $(R_0, R)$ such that, for any $r \in (R_0, R)$, $0 < d(r) < 1$ and $r(1 + d(r)) < R$. Moreover, $f(z) \neq 0$ for each $z$ such that $\abs{z-r} \leq r d(r)$ and every $r \in (R_0, R)$.
    \item For $k \geq 1$, define
    \[ A(z) = f'(z)/f(z) , \quad B_k(z) = z^kA^{(k-1)}(z)/k! , \quad B(z) = z B_1'(z)/2 .\]
    Then $B(r) > 0$ for $R_0 < r < R$ and $B_1(r) \to \infty$ as $r \to R^-$.
    \item Let $R_1 \in (R_0, R)$ be such that $B_1(R_1) > 0$. For any integer $n \geq B_1(R_1)$, let $r_n$ be the unique solution of the equation $B_1(r) = n+1$ with $r \in (R_1, R)$. For $j \geq 1$, define
    \[C_j(r) = -(B_{j+2}(r) + (-1)^{j}B_1(r)/(j+2))/B(r) .\]
    Suppose that there exist non-negative $D_n, E_n$ and $n_1 > 0$ such that the inequality $\abs{C_j(r_n)} \leq E_n D_n^j$ holds for all $n \geq n_1$ and $j \geq 1$.
    \item\label{cond:E} As $n \to \infty$, we have
    \[ B(r_n) d(r_n)^2 \to \infty , \quad D_n E_n B(r_n) d(r_n)^3 \to 0 , \quad D_n d(r_n) \to 0 .\]
\end{enumerate}

Harris and Schoenfeld show that under these conditions, the coefficients $\alpha_n$ have an asymptotic expansion as $n \to \infty$ of the following form.

\begin{theorem}[\cite{harris1968asymptotic}, Theorem 1]

    Suppose that $f(z)$ satisfies conditions \textup{(A)--(E)}. Then for all $N \geq 0$, there exist constants $K_N, L_N > 0$ such that for all $n \geq K_N$, we have
    \[
        \abs{
            \frac{\alpha_n}{f(r_n) / (2 r_n^n \sqrt{\pi B(r_n)})} 
            - 1 - \sum_{k = 1}^N \frac{F_k(n)}{B(r_n)^k}
        } \leq
        L_N \cdot \varphi_N(n),
    \]
    where
\begin{itemize}
    \item $F_k(n)$ are explicitly computable in terms of $C_j(r_n)$ for $1 \leq j \leq 2k$,
    \item $\varphi_N(n)$ is explicitly computable as a function of $D_n, E_n, B(r_n), d(r_n)$, and the maximum value of\, $\abs{f(z)}$ for $z$ on some curve depending on $r_n, d(r_n)$.
\end{itemize}
\end{theorem}

Müller verifies that $\exp(P(z))$ for $P(z)$ as above meets the conditions (A)--(E) by explicit verification. All implicit dependencies are absolute, and all dependencies on the polynomial $P(z)$ only are denoted as $O_P(1)$ or $\ll_P$.

\begin{theorem}[\cite{muller1997finite}, Lemma 1]
Let $P(z)$ be a polynomial satisfying \cref{eq:polynomial_coefficient_conditions}. Then the function $f(z) = \exp(P(z))$ satisfies conditions \textup{(A)--(E)} with:
\begin{itemize}
    \item[\textup{(A)}] Take $R = \infty$. 
    \item[\textup{(B)}] Take $R_0$ to be the maximum of $1$ and the largest real root of $B(r)$.
    \item[\textup{(D)}] Take $n_1 \gg_P 1$, $E_n \ll_P 1$, and $D_n = 1$.
    \item[\textup{(E)}] Take $d(r) = r^{-2q/5}$.
\end{itemize}
\end{theorem}

For our application, we also need explicit bounds on $K_N$, $L_N$ and $\varphi_N(n,d)$ in terms of $N$. We derive these from the proofs in \cite{harris1968asymptotic} and \cite{muller1997finite}.

\begin{proposition}
Let $P(z)$ be a polynomial satisfying \cref{eq:polynomial_coefficient_conditions}. Then
\[
    \varphi_N(n) \ll O_P(1)^{N+1} n^{-N-1}, \quad
    K_N = O_P(1) N^{O_P(1)}, \quad
    L_N = (O_P(1) N)^{O(1) N}.
\]
\end{proposition}
\begin{proof}
Let us first deal with $\varphi_N(n)$. We follow \cite[Lemma~2]{muller1997finite}. We have
    \[
        \varphi_N(n) =
        \max \left\{ X(n), Y(n), Z(n) \right\},
    \]
    where the three terms are defined and bounded as follows. The first term is, for some explicit constant $q_0$ depending only on $P$,
    \[
    X(n) = 
    \left(\frac{q_0}{\sqrt{B(r_n)}}\right)^{2N+2} \leq
    O_P(1)^{N+1} r_n^{-q(N+1)} \leq
    O_P(1)^{N+1} n^{-(N+1)},
    \]
    We have used $B(r_n) \geq r_n^q$ in the first inequality and $n < B_1(r_n) \ll_P r_n^q$ in the second. For the second term, we first have
    \[
    Y(n) = \frac{\exp(-B(r_n) r_n^{-4q/5})}{r_n^{-2q/5} \sqrt{B(r_n)}} \leq
    r_n^{9q/10} \exp(- r_n^{q/5}) \leq 
    (2n)^{9/10} \exp(- (2n)^{1/5}),
    \]
    where we have again used $B(r_n) \geq r_n^q$ and $2n \geq B_1(r_n) \geq r_n^q$, and assumed $n$ is large enough so that $r_n > 2000$. Taking $n \gg N^{10}$ and using $e^{-t} \leq t^{-k}$ for $t \geq k^2$ with $t = (2n)^{1/5}$ and $k = 5N+10$, we thus obtain
    \[
    Y(n) \ll (2n)^{9/10 - (5N+10)/5} \ll n^{-(N+1)}.
    \]
    Finally, for the third term, we have, for all $n \gg_P 1$,
    \[
        Z(n) = 
        \lambda(r_n) \sqrt{B(r_n)} \leq 
        \max \left\{ 
            \exp(-((q-1)/2)^2 q^{-1} r_n^{q/5}),
            \exp(-A r_n)
         \right\},
    \]
    where $A \gg_P 1$. Using $2n \geq B_1(r_n) \geq r_n^q$, we thus obtain
    \[
    Z(n) \leq
    \max \left\{ 
        \exp(- (2n)^{1/5}/8),
        \exp(- A (2n)^{1/q})
        \right\}.
    \]
    Taking $n \gg_P N^{\max \{10, 2q \}}$, we can again bound the exponentials as in the previous case to obtain $Z(n) \ll n^{-(N+1)}$.

Let us now turn to the constants $K_N$ and $L_N$. From the argument above, we see that we must certainly take $K_N \gg_P N^{O_P(1)}$. Other restrictions on $K_N$, $L_N$ arise from:
\begin{itemize}
    \item condition (D) in \cite{harris1968asymptotic}, where we need $n \gg_{P} 1$, 
    \item condition (E) in \cite{harris1968asymptotic} in:
        \begin{itemize}
            \item bounding $\abs{Y_3} \leq \frac{1}{2}$, where we need $B(r_n) d(r_n)^3 \ll_q 1$, thus $n \gg_P 1$,
            \item equation (20) for bounding $V_p(\varphi) \ll \beta_n \varphi^2 (e \varphi)^{2p} + (e E_n \beta_n \varphi^3)^{2p}$, where we need $1 - e \abs{\varphi} \ll 1$, thus $n \gg_P 1$,
            \item bounding $\Delta_n \geq 6N$, where we need $n \gg_P (6N)^5$,
            \item bounding $V_1(\delta_n) \leq 1$, where we need $n \gg_P 1$,
        \end{itemize}
    \item equation (16) in \cite{harris1968asymptotic} for bounding $\abs{Y_4} \ll O_P(1)^{N+1} Y_5$,
    \item just after equation (20) for bounding $V_{N+1}(\varphi) \ll O(1)^{N+1} Y_5$,
    \item bounding $Y_6 \ll E_n \beta_n^{-N-3/2} (E_n^{2N+1} + 1) \cdot \Gamma(3N + 7/2)$, where $\Gamma$ is Euler's gamma function that can be bounded as $\Gamma(3N+7/2) \leq (3N+4)! \leq (7N)^{7N}$. 
\end{itemize}
Overall, we can thus take $K_N = O_P(1) N^{O_P(1)}$ and $L_N = (O_P(1) N)^{O(1) N}$.
\end{proof}

Let 
\[
    \mathcal M(r) = \frac{\exp(P(r))}{2 r^n \sqrt{\pi B(r)} }, \quad
    \mathcal S_N(r) = 1 + \sum_{k=1}^N \frac{P_k(r)}{B(r)^{3k}},
\]
where $P_k(r)$ are as in \cite[Proposition 1]{muller1997finite}. It follows from all the above that the asymptotic expansion from Harris and Schoenfeld reduces to the following.

\begin{theorem}[quantitative form of \cite{muller1997finite}, Proposition 1]
    \label{thm:quantitative_muller_rn}
    Let $P(z)$ be a polynomial satisfying \cref{eq:polynomial_coefficient_conditions}. Then for all $N \geq 0$ and $n \gg_P N^{O_P(1)}$, we have
    \[
        \alpha_n = \mathcal M(r_n) \left( \mathcal S_N(r_n) + E_{n,N} \right), 
        \quad
        \abs{E_{n,N}} \leq (O_P(1)N)^{O(1)N} n^{-N-1}.
    \]
\end{theorem}

\subsection{Approximating $r_n$}

In order to convert the asymptotic expansion above into an asymptotic expansion in terms of $n$, we need to expand $r_n$ in terms of $n$ as in \cite[Section 3.6]{muller1997finite} and then use this to expand all the other terms in the asymptotic expansion from the previous section.

\subsubsection{Expanding $r_n$}

Define the polynomial
\[
Q(z) = z^{q-1} P'(1/z) = \sum_{i = 1}^q i c_i z^{q - i}
\] 
with $Q(0) = 1$. Let $\phi(z) = Q(z)^{1/q}$. This is an analytic function on a disc centered at $0$ with $\phi(0) = 1$. Thus the function $z \mapsto z/\phi(z)$ is also analytic with derivative $1$ at $z = 0$, and hence it is invertible in a neighborhood of $0$. In particular, there are $0 < \epsilon, \delta < 1$ such that for every $w$ with $\abs{w} < \epsilon$, the equation $w \phi(z) = z$ has a unique solution in the domain $\abs{z} < \delta$. Moreover, by Lagrange inversion, this solution can be expressed as a power series
\[
    z = \sum_{\nu = 1}^\infty \frac{\beta_\nu}{\nu} w^\nu,
\]
where the $\beta_\nu$ can be computed explicitly in terms of $P$ (see \cite[formula (17)]{muller1997finite}).

Let us use the above with $w = (n+1)^{-1/q}$ and write $z = r^{-1}$. For $n \gg_\epsilon 1$, the equation $r P'(r) = n+1$ therefore has a unique solution $r_n$ in the domain $\abs{r} > 1/\delta$, and this solution is given by

\[
    \frac{1}{r_n} = \sum_{\nu = 1}^\infty \frac{\beta_\nu}{\nu} (n+1)^{-\nu/q}
    = n^{-1/q} \left( 1 + \sum_{i = 1}^\infty \gamma_i n^{-i/q} \right) ,
\]
where we obtained the coefficients $\gamma_i$ as in \cite[formula (24)]{muller1997finite} by expanding $(n+1)^{-\nu/q} = n^{-\nu/q} (1 + n^{-1})^{-\nu/q}$. Thus
\[
    r_n = n^{1/q} \left(1 + G(n^{-1/q})\right)^{-1}, \qquad
    G(w) = \sum_{i = 1}^\infty \gamma_i w^{i}.
\]
The function $G(w)$ is also analytic on the disc centered at $0$ with radius $\epsilon$. By Cauchy's estimate, it follows that the coefficients $\gamma_i$ satisfy $\abs{\gamma_i} \leq O_P(1) \epsilon^{-i}$.

\subsubsection{Approximating $r_n$ by $\rho_n(s)$}

For $s \geq 1$, let
\[
    \rho_n(s) = n^{1/q} \left( 1 + G_{s}(n^{-1/q}) \right)^{-1}, \qquad
    G_{s}(w) = \sum_{i = 1}^{q+s-1} \gamma_i w^{i}.
\]
Thus $G_{s}$ is the truncation of $G$ to its first $q+s-1$ terms. The value $\rho_n(s)$ is well defined for $n \gg_P 1$.

\begin{lemma}\label{lemma:rn_vs_rhon}
For all $s \geq 1$ and all $n \gg_P 1$,
\[
    \abs{r_n - \rho_n(s)} \ll_P (\epsilon/2)^{-s} n^{-(q+s-1)/q}.
\]
\end{lemma}
\begin{proof}
Observe that
\[
    \abs{G(n^{-1/q}) - G_{s}(n^{-1/q})} \leq
    O_P(1) n^{-(q+s)/q} \sum_{i = q+s}^\infty \epsilon^{-i} n^{-(i - q - s)/q}.
\]
Since $n^{-1/q} < \epsilon/2$ for $n \gg_P 1$, the above is at most
\[
    O_P(1) (\epsilon/2)^{-s} n^{-(q+s)/q}.
\]
In particular, since we have the \emph{a priori} bound 
\[
\abs{G(n^{-1/q})} \ll_P n^{-1/q},\]
it follows that for $n \gg_P 1$,
\[\abs{G_{s}(n^{-1/q})} 
\ll_P n^{-1/q} .
\]
Now, the function $x \mapsto (1+x)^{-1}$ is Lipschitz for $\abs{x} < \delta$ with constant $1/(1-\delta)^2$. Thus, for $n \gg_P 1$, we obtain
\[
    \abs{r_n - \rho_n(s)} \leq
    \frac{n^{1/q}}{(1 - \delta)^2} \cdot
    O_P(1) (\epsilon/2)^{-s} n^{-(q+s)/q} =
    O_P(1) (\epsilon/2)^{-s} n^{-(q+s-1)/q}. \qedhere
\]
\end{proof}

\subsubsection{Replacing $r_n$ by $\rho_n(s)$ in \Cref{thm:quantitative_muller_rn}}

\begin{theorem}[quantitative form of \cite{muller1997finite}, formula (33)]
    \label{quantitative_muller_rhon}
    For all $s \geq 1$, $N \geq 0$ and $n \gg_P N^{O_P(1)}$,
    \[
    \alpha_n = \mathcal M(\rho_n(s)) \left( \mathcal S_N(\rho_n(s)) + E_{n,N,s} \right) ,
    \]
    where
    \[
    \abs{E_{n,N,s}} \ll
    (O_P(1) N)^{O(1) N} n^{-N-1} + 
    (O(1)N)^{O(1)N} O_P(1)^s n^{-s/q}.
    \]
\end{theorem}

We need three preliminary lemmas that will combine to prove the above theorem.

\begin{lemma}
For all $s \geq 1$ and $n \gg_P 1$,
\[
    \abs{\frac{\alpha_n}{\mathcal M(r_n)}} \ll 1.
\]
\end{lemma}
\begin{proof}
As the approximation improves for larger $N$, we can choose $N = 0$ in \Cref{thm:quantitative_muller_rn}. Therefore $\alpha_n = \mathcal M(r_n) (1 + E_{n,0})$ with $\abs{E_{n,0}} \ll_P n^{-1}$. For $n \gg_P 1$, we thus have $\abs{\alpha_n / \mathcal M(r_n)} \leq 2$.
\end{proof}

\begin{lemma}
For all $s \geq 1$ and $n \gg_P 1$,
\[
\abs{\frac{\mathcal M(r_n)}{\mathcal M(\rho_n(s))} - 1} \ll_P (\epsilon/2)^{-s} n^{-s/q}.
\]
\end{lemma}
\begin{proof}
Define
\[\Psi(r) = \log \mathcal M(r) = P(r) - n \log(r) - \frac{1}{2} \log(B(r)) + O(1) . \]
Let $I_{n,s}$ be the interval between $r_n$ and $\rho_n(s)$. By the mean value theorem, we have
\[
    \abs{\Psi(r_n) - \Psi(\rho_n(s))} \leq
    \sup_{r \in I_{n,s}} \abs{\Psi'(r)} \cdot
    \abs{r_n - \rho_n(s)}.
\]
Now, $\Psi'(r) = P'(r) - n/r - B'(r)/(2 B(r))$. For $n \gg_P 1$ we have $r_n, \rho_n(s)  \sim n^{1/q}$, and so we deduce that for $r \in I_{n,s}$, we have
\[
    \abs{P'(r)} \ll_P n^{(q-1)/q}, \quad
    \abs{n/r} \ll_P n^{(q-1)/q}, \quad
    \abs{\frac{B'(r)}{B(r)}} \ll_P n^{-1/q},
\]
as long as $n \gg_P 1$. The derivative $\Psi'(r)$ is thus bounded by $O_P(1) n^{(q-1)/q}$ for $r \in I_{n,s}$. Lemma~\ref{lemma:rn_vs_rhon} implies that
\[
    \abs{\Psi(r_n) - \Psi(\rho_n(s))} \ll_P
    (\epsilon/2)^{-s} n^{-s/q}.
\]
Finally, since $\abs{e^x - 1} \leq 2 \abs{x}$ for $\abs{x} \leq 1$, it follows that
\[
    \abs{\exp \left(  \Psi(r_n) - \Psi(\rho_n(s)) \right) - 1} \ll_P
    (\epsilon/2)^{-s} n^{-s/q},
\]
as required.
\end{proof}

\begin{lemma} \label{lemma:Sn_rn_vs_rhon}
For all $s \geq 1$ and $n \gg_P 1$,
\[
    \abs{\mathcal S_N(r_n) - \mathcal S_N(\rho_n(s))} \ll_P (O(1)N)^{O(1)N} (\epsilon/2)^{-s} n^{-s/q}.
\]
\end{lemma}
\begin{proof}
By definition, $\mathcal S_N(r)-1$ is a sum of $N$ terms of the form $P_k(r)/B(r)^{3k}$ for $1 \leq k \leq N$. Following \cite[Proposition 1]{muller1997finite}, each $P_k(r)$ is a polynomial in $r$, given as a sum by
\[
    P_k(r) = \frac{(-1)^k}{2^{2k-1}} \sum_{i = 1}^{2k} (-1/4)^i A_{k,i} B(r)^{2k-i} V_{k,i}(r).
\]
Here, $A_{k,i}$ are constants depending only on $k$ and $i$ that can be upper bounded by $(O(1)k)^{O(1)k}$. The summand $V_{k, i}(r)$ is a polynomial in $r$ given as the sum
\[
V_{k,i}(r) = \sum_{\substack{j_1 + \cdots + j_i = 2k \\ j_1, \dots, j_i \geq 1}} W_{j_1}(r) \cdots W_{j_i}(r) ,
\]
where $W_j(r)$ are polynomials in $r$ of degree at most $q$ with coefficients bounded in absolute value by $O_P(1)$. As in the previous lemma, let $I_{n,s}$ be the interval between $r_n$ and $\rho_n(s)$. Thus, for $n \gg_P 1$ and $r \in I_{n,s}$, we have
\[
\abs{W_j(r)} \ll_P n, \quad
\abs{V_{k,i}(r)} \ll_P 
O_P(1)^k \binom{2k-1}{i - 1} n^{2k} \ll_P 
O_P(1)^k n^{2k}.
\]
Since $B(r) \sim_P r^q$ and $B'(r) \ll_P r^{q-1}$ for $r \in I_{n,s}$, it follows that
\begin{equation}
    \label{eq:Pk_bounds}
    \abs{P_k(r)} \ll_P (O(1)k)^{O(1)k} r^{2k q},
    \quad
    \abs{P'_k(r)} \ll_P (O(1)k)^{O(1)k} r^{2k q - 1}.
\end{equation}
Now we can bound the derivative of $\mathcal S_N(r)$ for $r \in I_{n,s}$ as
\[
    \abs{\mathcal S'_N(r)} =
    \abs{\sum_{k = 1}^N \left( 
        \frac{P'_k(r)}{B(r)^{3k}} - k \frac{P_k(r) B'(r)}{B(r)^{3k+1}}
    \right)}
    \ll \sum_{k=1}^N (O(1)k)^{O(1)k} r^{-kq - 1}.
\]
As long as $n \gg_P 1$, this implies that $\abs{\mathcal S'_N(r)}$ is at most $(O(1)N)^{O(1)N} r^{-q - 1}$ for $r \in I_{n,s}$. It now follows from the intermediate value theorem and \Cref{lemma:rn_vs_rhon} that
\[
    \abs{\mathcal S_N(r_n) - \mathcal S_N(\rho_n(s))} 
    \ll_P
    (O(1)N)^{O(1)N} r_n^{-q - 1} \cdot 
    (\epsilon/2)^{-s} n^{-(q+s-1)/q}.
\]
The latter is $(O(1)N)^{O(1)N} O(1)^s n^{-s/q}$ for $n \gg_P 1$, completing the proof.
\end{proof}

\begin{proof}[Proof of Theorem~\ref{quantitative_muller_rhon}]
Bound the difference between the error term at $r_n$ and $\rho_n(s)$ as
\[
\abs{E_{n,N,s} - E_{n,N}} \leq
\abs{\frac{\alpha_n}{\mathcal M(r_n)}} \cdot \abs{\frac{\mathcal M(r_n)}{\mathcal M(\rho_n(s))} - 1} +
\abs{\mathcal S_N(r_n) - \mathcal S_N(\rho_n(s))}.
\]
The stated bound for $E_{n,N,s}$ now follows from the three previous lemmas combined with the bound for $E_{n,N}$ in \Cref{thm:quantitative_muller_rn}.
\end{proof}

We now select a specific value $N = \lceil s/q \rceil - 1$ in the previous theorem to balance the two error terms. With this choice, we have $n^{-N-1} \leq n^{-s/q}$. Consequently, \Cref{quantitative_muller_rhon} yields the following.

\begin{corollary}
    \label{cor:quantitative_muller_rhon}
    For all $s \geq 1$ and $n \gg_P s^{O_P(1)}$,
    \[
    \alpha_n = \mathcal M(\rho_n(s)) \left( \mathcal S_{\lceil s/q \rceil - 1}(\rho_n(s)) + E_{n,s} \right), \quad
    \abs{E_{n,s}} \leq (O_P(1) s)^{O(1) s} n^{-s/q}.
    \]
\end{corollary}

\subsection{Quantitative asymptotics in terms of $n$}

We are now ready to perform the final step of converting the asymptotic expansion from the previous section into an asymptotic expansion in terms of $n$ alone. Following \cite[formula (26)]{muller1997finite}, for an explicit constant $K(P)$ depending only on the polynomial $P$, let
\[
    \Pref(u) = \frac{K(P)}{2 \sqrt{\pi}} u^{n + q/2} \exp(P(u^{-1}))
\]
be the prefactor in the asymptotic formula for $\alpha_n$ from \cite{muller1997finite} with $u = n^{-1/q}$. Write $\Pref_n = \Pref(n^{-1/q})$. 

\begin{lemma}\label{lemma:XYhol}
Let $u$ be a complex variable, and let
\[
  \rho(u) = u^{-1}(1+G_s(u))^{-1}, \quad
  X_s(u) = \frac{\mathcal M(\rho(u))}{\Pref(u)}, \quad
  Y_s(u) = \mathcal S_{\lceil s/q\rceil-1}(\rho(u)).
\]
There exist $\epsilon > 0$ and $M > 0$ depending only on $P$ such that for every $s \geq 1$:
\begin{enumerate}
\item $X_s$ is holomorphic on $0<|u|<\epsilon$ with
\[ \sup_{0<|u|<\epsilon}|X_s(u)| \leq M ;\]
\item $Y_s$ is holomorphic on $0<|u|<\epsilon$ with
\[ \sup_{0<|u|<\epsilon}|Y_s(u)| \leq (Ms)^{Ms} ;\]
\item Both $X_s$ and $Y_s$ extend holomorphically to $|u|<\epsilon$ with $X_s(0) = Y_s(0) = 1$.
\end{enumerate}
\end{lemma}
\begin{proof}
    Taking logarithms, we obtain
    \[
        \big|\log X_s(u)\big| \ll 
        \big| P(\rho(u)) - P(u^{-1}) - n \big(\log\rho(u) + \log u\big) \big| +
        \big| \log B(\rho(u)) + q \log u \big|.
    \]
    Since $\log\rho(u) = -\log u - \log(1+G_s(u))$ and $G_s$ is analytic on $\{|u|<\epsilon\}$ with $G_s(0)=0$, the product $\rho(u)u$ is bounded by $O_P(1)$ on that disc. As $B$ is a polynomial of degree $q$ with leading coefficient $q/2$, the quotient $B(\rho(u))/u^{-q}$ is bounded by $O_P(1)$ (from above and below) as well. Hence $\log B(\rho(u)) = - q\log u + O_P(1)$ and the second summand is bounded by $O_P(1)$ for $|u|<\epsilon$.

    It remains to bound the first summand. By \cite[Section 3.6]{muller1997finite} the coefficients $\gamma_i$ of $G_s(u)$ vanish for $i < q/2$, and hence we may write $G_s(u)=u^{q/2}\Delta_s(u)$ with $\Delta_s$ analytic on $\{|u|<\epsilon\}$ and bounded by $O_P(1)$ there. Expanding,
    \begin{equation} \label{eq:Prhou_minus_Puinv}
        P(\rho(u)) - P(u^{-1}) =
        \sum_{i=1}^q c_i u^{-i}\big[(1+G_s(u))^{-i}-1\big].
    \end{equation}
    Since $G_s(u)=O_P(u^{q/2})$ on the disc, we obtain, by Taylor expansion,
    \[
        (1+G_s(u))^{-i}-1 = - i G_s(u) + O_P(G_s(u)^2) = O_P(u^{q/2}).
    \]
    For $i<q$ any nonzero contribution in \cref{eq:Prhou_minus_Puinv} can occur only when $i\le q/2$. Then
    \[
        u^{-i}O_P(u^{q/2}) = O_P\big(u^{q/2-i}\big),
    \]
    which is uniformly bounded on $\{|u|<\epsilon\}$ because the exponent is non-negative. For the top degree $i=q$, we must treat the terms more carefully. Using $c_q=1/q$ and $u^q=n^{-1}$, the relevant part reduces to
    \[
        c_q n\big((1+G_s(u))^{-q}-1\big) + n\log(1+G_s(u)).
    \]
    Expanding the power as before and the logarithm for small $G_s(u)$ (shrinking $\epsilon$ if necessary so that $|G_s(u)| \ll 1/q$), we get
    \[
        c_q n\big(-qG_s(u) + O_P(G_s(u)^2)\big) + n\big(G_s(u) + O_P(G_s(u)^2)\big).
    \]
    The linear terms cancel, because $c_q=1/q$, leaving a remainder of order $nG_s(u)^2$. Since $G_s(u)=O_P(u^{q/2})$, it follows that $nG_s(u)^2=O_P(nu^q)=O_P(1)$.

    Combining the bounds above shows that $\log X_s(u)$ is bounded by $O_P(1)$ on $0<|u|<\epsilon$, so each $X_s$ is holomorphic and bounded there. Therefore $X_s$ extends holomorphically across $0$ by the Riemann removable singularity theorem with the value $X_s(0)=1$, as asserted in \cite[Section 3.6]{muller1997finite}.

    The argument for $Y_s$ is simpler. By definition, $Y_s(u)$ is a sum of terms of the form $P_k(\rho(u))/B(\rho(u))^{3k}$ for $k \leq \lceil s/q \rceil - 1$. On the disc $|u|<\epsilon$, the functions $\rho(u)u$ and $B(\rho(u))/u^{-q}$ are bounded by $O_P(1)$, as shown above. Each $P_k(\rho(u))$ is bounded by $(O(1)k)^{O(1)k} \rho(u)^{2k q}$ by \cref{eq:Pk_bounds}. Thus each summand is bounded by
    \[
        \abs{\frac{P_k(\rho(u))}{B(\rho(u))^{3k}}} \ll_P
        (O(1)k)^{O(1)k} u^{k q}
    \]
    on the disc of radius $\epsilon$. Summing over $k$ yields the bound $(O_P(1)s)^{O(1)s}$. As every nonconstant summand vanishes in $0$, the extended value at $0$ is $Y_s(0)=1$. 
\end{proof}

\begin{theorem}[quantitative form of \cite{muller1997finite}, Theorem 2]
    \label{thm:quantitative_muller_final}
    For all $s \geq 1$ and $n \gg_P s^{O_P(1)}$,
    \[
    \frac{\alpha_n}{\Pref_n} = 1 + \sum_{\nu = 1}^{s-1} \mathcal C_\nu n^{-\nu/q} + R_{n,s},
    \]
    where the coefficients $\mathcal C_\nu$ are defined as in \cite[formula (25)--(27)]{muller1997finite}, and the remainder term satisfies
    \[
    \abs{R_{n,s}} \leq O_P(s)^{O_P(s)} n^{-s/q}.
    \]
\end{theorem}
\begin{proof}
Let $Z_s(u) = X_s(u) Y_s(u)$. By the previous lemma, $Z_s$ is holomorphic on $\{ |u|<\epsilon \}$ and bounded by $(O_P(1)s)^{O(1)s}$ there. Thus $Z_s$ has a Taylor expansion
\[
    Z_s(u) = \sum_{\nu = 0}^\infty c_{s,\nu} u^\nu
\]
on the disc $\{ \abs{u} < \epsilon \}$ with $c_{s,0} = 1$. Moreover, for $\abs{u} < \epsilon/2$,
\[
    \abs{Z_s(u) - \sum_{\nu = 0}^{s-1} c_{s,\nu} u^\nu}
    \leq
    \frac{\sup_{0 < |u| < \epsilon} |Z_s(u)|}{1 - |u|/\epsilon} (|u|/\epsilon)^s
    \leq
    (O_P(1)s)^{O(1)s} |u|^s .
\]
Set $u = n^{-1/q}$, and let $n \gg_P s^{O_P(1)}$. By Corollary~\ref{cor:quantitative_muller_rhon}, we have
\[
    \frac{\alpha_n}{\Pref_n} =
    Z_s(n^{-1/q}) + X_s(n^{-1/q}) E_{n,s}.
\]
Since $\abs{E_{n,s}} \le (O_P(1)s)^{O(1)s} n^{-s/q}$, the error term $E_{n,s}$ and the tail of the Taylor expansion of $Z_s(n^{-1/q})$ are of the same order, while $X_s$ is bounded for $n \gg_P 1$. Hence
\[
    \abs{\frac{\alpha_n}{\Pref_n} - \sum_{\nu = 0}^{s-1} c_{s,\nu} n^{-\nu/q}}
    \ll (O_P(1) s)^{O(1) s} n^{-s/q}. 
\]
Since this expansion holds for every $s$ with error $O_{P,s}(n^{-s/q})$, the coefficients $c_{s,\nu}$ must in fact be independent of $s$, and they must equal the coefficients $\mathcal C_\nu$ from \cite[formula (25)--(27)]{muller1997finite}.
\end{proof}

\section{Fractional expansions for the number of group actions}
\label{sec:fractional_expansions}

Let $(\beta_n)_{n \geq a}$ be a sequence of real numbers. We say that $(\beta_n)_{n \geq a}$ admits a \emph{fractional expansion with parameters} $q,A,B,C,D,E \geq 1$, in symbols $\FE_q(A,B,C,D,E)$, if there exist real coefficients $(b_k)_{k \geq 0}$ with zeroth coefficient $b_0 = 1$ such that, for all $s \geq 1$ and for all $n \geq A s^B$, we have
\[
    \abs{\beta_n - \sum_{k=0}^{s-1} b_k n^{-k/q}} 
    \leq
    C (Ds)^{Es} n^{-s/q}.
\]
Note that $A \ge a$ is required for this definition to make sense. 

A crucial consequence of \Cref{thm:quantitative_muller_final} is that the sequence of normalized coefficients $(\alpha_n/\Pref_n)_{n \geq 0}$ admits $\FE_q$ with all parameters $O_P(1)$. We will now establish a series of lemmas showing that the class of sequences admitting $\FE_q$ is closed under various operations. This will allow us to transform the sequence $(\alpha_n/\Pref_n)_{n \geq 0}$ into the sequence required for verifying \Cref{requirement:asymptotic_expansion}.

\subsection{Operations with asymptotic expansions}

\begin{lemma}[coefficients bound]
\label{lemma:coeffBound}
Let the sequence $(\beta_n)_{n \geq a}$ admit $\FE_q(A,B,C,D,E)$ with coefficients $(b_k)_{k \geq 0}$. Then, for all $k \geq 1$, the coefficients satisfy
\[
\abs{b_k} \leq 2C (2 D k)^{2E k}.
\]
\end{lemma}

\begin{proof}
Take $n \geq A (k+1)^B$. Apply the fractional expansion estimate with $s = k$ and $s = k+1$, and subtract them to obtain
\[
    \abs{b_k} n^{-k/q} 
    \leq
    C (Dk)^{Ek} n^{-k/q} + C (D(k+1))^{E(k+1)} n^{-(k+1)/q}.
\]
Multiplying by $n^{k/q}$ gives $\abs{b_k} \leq 2C (2Dk)^{2Ek}$, as desired.
\end{proof}

\begin{lemma}[closure properties of $\FE$: product]
\label{lemma:clProduct}
Let the sequences $(\beta_n)_{n \geq a_1}$ and $(\gamma_n)_{n \geq a_2}$ admit $\FE_q(A_1,B_1,C_1,D_1,E_1)$ and $\FE_q(A_2,B_2,C_2,D_2,E_2)$, respectively. Then the product sequence $(\beta_n\gamma_n)_{n \geq \max(a_1,a_2)}$ admits
    \[
    \FE_q \left(
    \max(A_1, A_2), \,
    \max(B_1, B_2), \,
    2 C_1 C_2, \,
    2 \max(D_1, D_2), \,
    18 \max(E_1, E_2)
    \right).
    \]
\end{lemma}
\begin{proof}
Assume that $(\beta_n)_{n \geq a_1}$ and $(\gamma_n)_{n \geq a_2}$ admit a fractional expansion with coefficients $(b_k)_{k \geq 0}$ and $(c_k)_{k \geq 0}$, respectively. Set
\[ A = \max(A_1, A_2), \quad B = \max(B_1, B_2), \quad D = \max(D_1, D_2), \quad E = \max(E_1, E_2). \]
For $n \geq A s^B$, write
\[
\beta_n = M_\beta + E_\beta,
\quad
\gamma_n = M_\gamma + E_\gamma,
\]
where
\[
M_\beta = \sum_{i=0}^{s-1} b_i n^{-i/q}, \quad
M_\gamma = \sum_{i=0}^{s-1} c_i n^{-i/q}, \quad
\abs{E_\beta}, \abs{E_\gamma} \leq \max(C_1, C_2) (D s)^{E s} n^{-s/q}.
\]
By \Cref{lemma:coeffBound}, for every $i \leq s-1$, we have
\[ \abs{b_i}, \, \abs{c_i} \leq 2 \max(C_1, C_2) (2 D s)^{2 E s} .\]
Hence
\[ \abs{M_\beta}, \, \abs{M_\gamma} \leq
2 \max(C_1, C_2) s (2 D s)^{2 E s} .\]
Note that
\[
\beta_n \gamma_n = 
M_\beta M_\gamma
+
E_\beta M_\gamma
+
E_\gamma M_\beta
+
E_\beta E_\gamma.
\]
The sum of the last three terms is at most
\[
    3 \cdot C_1 C_2 \cdot (Ds)^{Es} n^{-s/q} \cdot 2 s (2Ds)^{2Es}
    \leq
    3 C_1 C_2 (2Ds)^{4Es} n^{-s/q}.
\]
For the main term, set
\[d_k = \sum_{i+j=k} b_i c_j .\]
Note that $d_0 = b_0 c_0 = 1$, and
\[ \abs{d_k} \leq (k+1) 4 C_1 C_2 (2Dk)^{4Ek} .\]
Then
\[
    M_\beta M_\gamma =
    \sum_{k=0}^{s-1} d_k n^{-k/q} +
    \sum_{k=s}^{2s-2} d_k n^{-k/q} ,
\]
where the second sum is at most
\[
s \cdot 8 s C_1 C_2 (4Ds)^{8Es} \cdot n^{-s/q} \leq 
2 C_1 C_2 (2Ds)^{18Es} n^{-s/q}.
\]
Thus, $\beta_n \gamma_n$ admits $\FE_q(A, B, 2 C_1 C_2, 2D, 18E)$ with coefficients $(d_k)_{k \geq 0}$.
\end{proof}

\begin{lemma}[closure properties of $\FE$: reciprocal]
\label{lemma:clReciprocal}
Let the sequence $(\beta_n)_{n \geq a}$ admit $\FE_q(A,B,C,D,E)$. Suppose that $\beta_n \neq 0$ for all $n \geq a$. Then the reciprocal sequence $(1/\beta_n)_{n \geq a}$ admits
\[
\FE_q \left(
\max (A,(2CD^E)^q), \,
B, \,
6, \,
8CD, \,
6E
\right).
\]
\end{lemma}
\begin{proof}
Let $(b_k)_{k \geq 0}$ be the fractional expansion coefficients of $(\beta_n)_{n \geq 0}$.
Applying $\FE_q(A,B,C,D,E)$ with $s=1$ gives, for $n\geq A$,
\[
\abs{\beta_n - 1} \leq C D^E n^{-1/q} .
\]
Hence, if $n \geq (2CD^E)^q$, then $\abs{\beta_n} \geq 1/2$. Recursively define the coefficients $(e_k)_{k \geq 0}$ by $e_0 = 1$ and, for $k\ge 1$,
\[
e_k=-\sum_{j=1}^k b_j e_{k-j} .
\]
For $s \geq 1$, set
\[
B_s(x)=\sum_{k=0}^{s-1} b_k x^k,
\qquad
E_s(x)=\sum_{k=0}^{s-1} e_k x^k.
\]
It follows that $E_s(x)B_s(x)=1+x^s Q_s(x)$ for some polynomial $Q_s$ of degree at most $s-2$. Our aim is now to bound the sizes of these polynomials at $x=n^{-1/q}$. 

\noindent \emph{Step 1: bounds on $e_k$.}
Let $k \geq 1$. By \Cref{lemma:coeffBound} applied to $(\beta_n)$, we have $\abs{b_k} \leq 2C(2Dk)^{2Ek}$. We claim that, for $k\ge 1$,
\[
\abs{e_k} \leq (8CDk)^{2Ek} .
\]
Indeed, for $k=1$, this holds because
\[ 
\abs{e_1} = \abs{b_1} \leq 2C(2D)^{2E} \leq (8CD)^{2E}.
\]
For $k \geq 2$, bound $e_k$ using the recursion as
\[
\abs{e_k} \leq \abs{b_k} \abs{e_0} + \sum_{j=1}^{k-1} \abs{b_j} \abs{e_{k-j}}.
\]
The first term can be bounded by the coefficient bound lemma as
\[ 
\abs{b_k} \leq
2C(2Dk)^{2Ek} \leq
    \frac{(8CDk)^{2Ek}}{2}.
\]
The sum can be bounded using the inductive hypothesis by 
\[
    \sum_{j=1}^{k-1} 2C(2Dj)^{2Ej} \cdot (8CD (k-j))^{2E(k-j)}.
\]
Using $2C \leq (2C)^{2Ej}$ and $8CD(k-j) \leq 8CDk$, this is at most
\[
(8CDk)^{2Ek} \cdot \sum_{j=1}^{k-1} \left(\frac{j}{2k}\right)^{2Ej}
    \leq \frac{(8CDk)^{2Ek}}{2}.
\]
Hence $\abs{e_k} \leq (8CDk)^{2Ek}$, as claimed.

\noindent\emph{Step 2: bounds on $B_s,E_s,Q_s$.}
Let $x=n^{-1/q} \leq 1$. Then
\[
\abs{B_s(x)} \leq 
\sum_{k=0}^{s-1} \abs{b_k}
\leq s \cdot 2C(2Ds)^{2Es}
\le (8CDs)^{3Es},
\]
and similarly
\[
\abs{E_s(x)} \leq \sum_{k=0}^{s-1} \abs{e_k}
\leq 1 + s \cdot (8CDs)^{2Es}
\leq (8CDs)^{3Es}.
\]
Since $E_s(x)B_s(x)=1+x^sQ_s(x)$, we obtain
\[
\abs{x^s Q_s(x)}
\leq \abs{E_s(x)B_s(x)} + 1
\leq 2 (8CDs)^{6Es} .
\]

\noindent\emph{Step 3: approximation of $1/\beta_n$.}
Apply $\FE_q(A,B,C,D,E)$ at order $s$ to get, for $n \geq A s^B$,
\[
\beta_n=B_s(n^{-1/q})+F_s,\qquad \abs{F_s} \leq C(Ds)^{Es}n^{-s/q}.
\]
This yields
\[
\abs{\beta_n E_s(n^{-1/q})-1}
\le \abs{B_s(n^{-1/q})E_s(n^{-1/q})-1} + \abs{F_s} \abs{E_s(n^{-1/q})}.
\]
The first term is at most $n^{-s/q}|Q_s(n^{-1/q})| \leq 2 (8CDs)^{6Es} n^{-s/q}$, while the second term is at most
\[ C(Ds)^{Es} n^{-s/q} \cdot (8CDs)^{3Es} \leq (8CDs)^{4Es} n^{-s/q} .\]
Finally, since $|\beta_n|\ge 1/2$ for $n \geq \max\big(A,(2CD^E)^q\big)$, we obtain
\[
\abs{\frac{1}{\beta_n}-E_s(n^{-1/q})}
=\frac{|\beta_n E_s(n^{-1/q})-1|}{|\beta_n|}
\le 6 (8CDs)^{6Es} n^{-s/q}
\]
This yields
\[
    \FE_q(\max (A,(2CD^E)^q),B,6,8CD,6E)
\] 
for $(1/\beta_n)$ with coefficients $(e_k)_{k\ge 0}$.
\end{proof}

\begin{lemma}[closure properties of $\FE$: shift]
\label{lemma:clShift}
Let the sequence $(\beta_n)_{n \geq a}$ admit $\FE_q(A,B,C,D,E)$, and let $\ell \in \Z$ with $\ell \neq 0$. Then the shifted sequence $(\beta_{n+\ell})_{n \geq \max(a - \ell, 0)}$
admits
\[
\FE_q \left(
2 \max(A, \abs{\ell}), \,
B, \,
4C, \,
2 D \abs{\ell}, \,
5E
\right).
\]
\end{lemma}
\begin{proof}
Let $(b_k)_{k\ge 0}$ be the fractional expansion coefficients of $(\beta_n)_{n\ge 0}$. Suppose that $n \geq \max(2A, 2 \abs{\ell}) s^B$, so that $n+\ell \geq n/2 \geq A s^B$. In particular, $n+\ell \geq 0$, and we may apply $\FE_q(A,B,C,D,E)$ at index $n+\ell$. Taking into account that $(n + \ell)^{-s/q} \leq 2^{s/q} n^{-s/q}$, we obtain 
\[
\beta_{n+\ell}
=
\sum_{k=0}^{s-1} b_k (n+\ell)^{-k/q} + X,
\qquad |X|\le C(Ds)^{Es}(n+\ell)^{-s/q}
\le 2C (Ds)^{2Es}n^{-s/q}.
\]
Write
\[
(n+\ell)^{-k/q}=n^{-k/q}\left(1+\frac{\ell}{n}\right)^{-k/q}.
\]
Let $u=\ell/n$, so that $|u|\le 1/2$. We can expand the second factor as a power series in $u$ up to term $J = \lceil (s-1-k)/q \rceil$ as
\[
(1+u)^{-k/q}=\sum_{j=0}^{J}\binom{-k/q}{j}u^j+R_{J}(u).
\]
Using \[ \abs{ \binom{-k/q}{j} }\le (e(k+1))^j\]
and $k<s$, we obtain
\[
|R_{J}(u)|
\le \sum_{j=J+1}^\infty (es|u|)^j
\le 2(es|u|)^{J+1},
\]
where we have used that $n \geq 2 |\ell| s$ and hence $es|u|\leq 1/2$. Therefore,
\[
|R_{J}(\ell/n)|\le 2(3s)^{J+1}\left(\frac{|\ell|}{n}\right)^{J+1}.
\]
Substituting yields
\[
(n+\ell)^{-k/q}
=
\sum_{j=0}^{J}\binom{-k/q}{j}\ell^j n^{-(k+jq)/q}
+
2 (3s)^{J+1}|\ell|^{J+1} n^{-(k+q(J+1))/q}.
\]
Since $J$ was chosen so that $k + q(J+1) \geq s$ and $J+1 \leq s$, the second summand is at most $2(3s)^s|\ell|^s n^{-s/q}$. Multiplying by $b_k$ and summing over $0 \leq k \leq s-1$, we collect all terms of the form $n^{-r/q}$ for $0 \leq r < s$ into coefficients $\widetilde b_r$, and absorb all remaining tail terms into a single error bound. This gives
\[
\widetilde\beta_n=\sum_{r=0}^{s-1}\widetilde b_r n^{-r/q}+\widetilde X.
\]
Note that $\widetilde b_0 = 1$. By \Cref{lemma:coeffBound}, we have $|b_k|\le 2C(2Dk)^{2Ek}\le 2C(2Ds)^{2Es}$ for all $k \leq s-1$. Hence,
\[ \sum_{k=0}^{s-1}|b_k|\le 2Cs(2Ds)^{2Es},\]
and so the remainder term can be bounded as
\[
|\widetilde X|
\le 2 |\ell|^s(3s)^{s}n^{-s/q}\sum_{k=0}^{s-1}|b_k| + C(Ds)^{Es}n^{-s/q}
\le 4C (2D |\ell| s)^{5Es}n^{-s/q}.
\]
This completes the proof.
\end{proof}

\begin{lemma}[closure properties of $\FE$: sum]
    \label{lemma:clSum}
    Let $r \geq 2$. Suppose that for all $t \leq r$, the sequence $(n^{-e_t/q} \beta_n^{(t)})_{n \geq a}$ admits $\FE_q(A_t,B_t,C_t,D_t,E_t)$, where $e_t$ is an integer.
    Let $e = \max_{1 \leq t \leq r}  e_t$ and $T = \abs{\{t \mid 1 \leq t \leq r, \, e_t = e\}}$. Then the sequence of sums
    \[
    \left(\frac{n^{-e/q}}{T}  \sum_{t=1}^r \beta_n^{(t)}\right)_{n \geq a}
    \]
    admits
    \[
    \FE_q \left(
    \max_{1 \leq t \leq r} A_t, \,
    \max_{1 \leq t \leq r} B_t, \,
    2r \max_{1 \leq t \leq r} C_t, \,
    2 \max_{1 \leq t \leq r} D_t, \,
    2 \max_{1 \leq t \leq r} E_t
    \right).
    \]
\end{lemma}
\begin{proof}
Let $A = \max_{1 \leq t \leq r} A_t$, and define $B,C,D,E$ similarly. Let us write
\[ 
\widetilde \beta_n^{(t)} = n^{-e_t/q} \beta_n^{(t)} , \qquad \delta_t = e - e_t \geq 0.
\]
Let $(b_k^{(t)})_{k \geq 0}$ be the asymptotic expansion coefficients of $\widetilde \beta_n^{(t)}$. For any $s \geq 1$, our hypothesis gives, for $n \geq A s^B$,
\[
    \widetilde \beta_n^{(t)}
    = 
    \sum_{k=0}^{s-1} b_k^{(t)} n^{-k/q}
    + X_t, \qquad
    \abs{X_t} \leq
    C_t (Ds)^{Es} n^{-s/q}.
\]
Thus, since $e = \delta_t + e_t$,
\[
    n^{-e/q} \sum_{t = 1}^r \beta_{n}^{(t)}
    =
    \sum_{t = 1}^r n^{-\delta_t/q} \widetilde \beta_n^{(t)}
    =
    \sum_{t=1}^r \sum_{k=0}^{s-1}
    b_k^{(t)} n^{-(k+\delta_t)/q}.
\]
For any $p \geq 0$, let
\[ Y_p = \{ (t,k) \mid 1 \leq t \leq r, \, 0 \leq k \leq s-1, \, k + \delta_t = p \} .\]
The double sum can then be rewritten as
\[
\sum_{p = 0}^{s - 1 + \delta_{\max}} \left( \sum_{(t,k) \in Y_p} b_k^{(t)} \right) n^{-p/q}, \qquad
\delta_{\max} = \max_{1 \leq t \leq r} \delta_t.
\]
Split this sum into the main part where $p \leq s-1$, and the tail part where $p \geq s$. The latter can be bounded by Lemma~\ref{lemma:coeffBound} as
\[
\abs{\sum_{p=s}^{s - 1 + \delta_{\max}} \sum_{(t,k) \in Y_p} b_k^{(t)}} \leq
\sum_{t = 1}^r \sum_{k=0}^{s-1} \abs{b_k^{(t)}} 
\leq
s \cdot \sum_{t = 1}^r 2C_t (2Ds)^{2Es}
\leq
2rC \cdot (2Ds)^{2Es}.
\]
Note that the zeroth coefficient of the sequence is $ T \geq 1$. Upon dividing by $T$, we obtain the desired fractional expansion, which completes the proof.
\end{proof}

\subsection{Fractional expansion for quotients of $\alpha_n$}

We will extract $\FE_{q}$ from analyticity results in the previous section using the following general tool.

\begin{lemma}\label{lemma:holomorphic_FE}
Let $u$ be a complex variable. Let $F(u)$ be holomorphic on the disc $\{ |u|<\epsilon \}$ for some $\epsilon > 0$ with $\sup_{|u|<\epsilon} |F(u)| \leq C$. Suppose that $F$ has real values on $\R$ and $F(0) \neq 0$.
Then there exist real coefficients $(f_k)_{k\ge 0}$ with $f_0 \neq 0$ such that for all $s,q\ge 1$ and $n \geq (2/\epsilon)^q$, we have
\[ 
\abs{F(n^{-1/q}) - \sum_{k=0} ^{s-1} f_k n^{-k/q}} \leq 2C (1/\epsilon)^s n^{-s/q}
\]
In particular, the sequence $(f_0^{-1} F(n^{-1/q}))_{n \geq 0}$ admits
\[
\FE_q \left(
    (2/\epsilon)^q, \,
    1, \,
    2C/\abs{f_0}, \,
    1/\epsilon, \,
    1
    \right).
\]
\end{lemma}
\begin{proof}
The function $F$ has Taylor expansion
\[
F(u) = \sum_{k=0}^\infty f_k u^k
\]
and Cauchy's estimates yield $\abs{f_k} \leq C \epsilon^{-k}$. Hence, for all $s \geq 1$,
\[
    \abs{F(u) - \sum_{k=0} ^{s-1} f_k u^k} \leq
    C \frac{1}{1 - \abs{u}/\epsilon} \left( \frac{\abs{u}}{\epsilon} \right)^s.
\]
Thus for $u = n^{-1/q}$ with $n \geq (2/\epsilon)^q$, we have
\[
    \abs{F(n^{-1/q}) - \sum_{k=0}^{s-1} f_k n^{-k/q}} \leq
    2C (1/\epsilon)^{s} n^{-s/q} .
\]
Dividing the inequality by $\abs{f_0}$ gives the desired fractional expansion.
\end{proof}

\begin{example}
    \label{example:pochhammer_FE}
Let $k \geq 1$, and let $(n)_k = n (n-1) \cdots (n-k+1)$ be the Pochhammer symbol. Consider the sequence with terms $(n^k/(n)_k)_{n \geq k}$. Using $u = n^{-1/q}$, we can write
\[
    \frac{n^k}{(n)_k} 
    =
    \prod_{j=0}^{k-1} \frac{n}{n-j} 
    =
    \prod_{j=0}^{k-1} \left( 1 - j u^q \right)^{-1}.
\]
The last expression defines a holomorphic function on the disc $\{|u|< \epsilon \}$, where $\epsilon = 1/(4k^2)^{1/q}$. In fact, for all $0 \leq j \leq k-1$,
\[ \abs{j u^q} < k \epsilon^q < 1/(4k) < 1 .\] Moreover, on this disc, we can compute
\[ \abs{(1 - j u^q)} \geq 1 - \abs{j u^q} \geq 1 - 1/(4k), \]
meaning that the product is bounded by $(1 - 1/(4k))^{-k} \leq 4/3$. Hence, we may take $C=2$ in the lemma above with this choice of $\epsilon$, and $f_0 = 1$. It follows that the sequence $(n^k/(n)_k)_{n \geq k}$ admits 
\[
\FE_q \left(
    (2/\epsilon)^q,\,
    1,\,
    2C,\,
    1/\epsilon,\,
    1
\right)
=
\FE_q \left(
    O_q(k^2), \,
    1, \,
    4, \,
    O_q(k^2),\,
    1
\right) .
\]
\end{example}

We are now ready to deduce fractional expansion for the sequence of quotients of $\alpha_n$ by its shifts.

\begin{theorem}\label{thm:alphaSeq}
    For any $k \geq 0$, the sequence $(n^{k/q} \alpha_{n+k}/\alpha_n)_{n \geq 0}$ admits
    \[
 \FE_q \left(
    O_P(k), \,
    O_P(1), \,
    O_P(1), \,
    O_P(k), \,
    O_P(1)
    \right).
    \]
\end{theorem}
\begin{proof}
Write
\begin{equation}
\label{eq:product_as_prefactor_ratio_and_normalized_alpha_ratio}
    n^{k/q} \cdot \frac{\alpha_{n+k}}{\alpha_n} =
    n^{k/q} \frac{\Pref_{n+k}}{\Pref_n} \cdot
    \frac{\alpha_{n+k}/\Pref_{n+k}}{\alpha_n/\Pref_n}.
\end{equation}
By \Cref{thm:quantitative_muller_final}, the sequence with terms
\[
a_n = \alpha_n/\Pref_n
\]
for $n \geq 0$ admits a fractional expansion for some parameters that are all $O_P(1)$. Applying the shift lemma with $\ell=k$, the sequence $(a_{n+k})_{n\ge 0}$ admits
\[
\FE_q \left( 
    O_P(k), \,
    O_P(1), \,
    O_P(1), \,
    O_P(k), \,
    O_P(1)
    \right) ,
\]
and, applying the reciprocal lemma to $(a_n)_{n\ge 0}$, the sequence $(1/a_n)_{n\ge 0}$ admits a fractional expansion with parameters that are all $O_P(1)$. It now follows by an application of the product lemma that the sequence $(a_{n+k}/a_n)_{n\ge 0}$ admits
\[
\FE_q \left(
    O_P(k), \,
    O_P(1), \,
    O_P(1), \,
    O_P(k), \,
    O_P(1)
    \right).
\]
This deals with the second factor in \cref{eq:product_as_prefactor_ratio_and_normalized_alpha_ratio}. Let us now take care of the prefactor ratio as well. We have the identity
\[
    \log \left( n^{k/q} \frac{\Pref_{n+k}}{\Pref_n} \right) =
    - \frac{n+k+q/2}{q} \log (1 + k/n) +
    P((n+k)^{1/q}) - P(n^{1/q}).
\]
Setting $u=n^{-1/q}$, we have $1+k/n = 1+ku^q$ and $(n+k)^{1/q}=u^{-1}(1+ku^q)^{1/q}$.
Let us thus define, for a complex variable $u$,
\[
    L_k(u) = - \frac{u^{-q} + k + q/2}{q} \log(1 + k u^q) +
    \left( P\big(u^{-1} (1 + k u^q)^{1/q}\big) - P(u^{-1}) \right).
\]
We claim that $L_k(u)$ extends to a holomorphic function on the disc $\{|u|< k^{-1/q}/2 \}$. Write $z = k u^q$, so that $\abs{z} < 1/2^q$ on the disc. Let
\[
f(z) = \log(1+z)/z, \quad
g_i(z) = ((1+z)^{i/q} - 1)/z \quad (1 \leq i \leq q).
\]
The functions $f$ and $g_i$ are holomorphic on the disc $\{|z| < 1\}$, hence bounded on the disc $\{|z| < 1/2^q\}$. Let $C$ be their common bound there, a constant that depends only on $q$. We can rewrite $L_k(u)$ as
\[
    L_k(u) =
    - \frac{k}{q} f(z)
    +
    \frac{k + q/2}{q} \log(1+z)
    + 
    \sum_{i=1}^q c_i u^{-i} z g_i(z).
\]
The first summand is bounded by $k C / q$, and the second by $(k + q/2) C 2^{-q}/q$. Both are $O_{C,q}(k)$. We can bound the third summand as
\[
    \abs{\sum_{i=1}^q c_i u^{-i} z g_i(z)}
    \leq
    \sum_{i=1}^q \abs{c_i} \abs{u}^{-i} \abs{z} C
    =
    \sum_{i=1}^q \abs{c_i} \abs{u}^{q - i} C k.
\]
On the disc $\{ \abs{u} < k^{-1/q}/2 \}$, we have $\abs{u}^{q - i} \leq k (k^{-1/q}/2)^{q - i} \leq 1$, hence the above is at most $O_{P,C}(k)$. Moreover, as $u \to 0$, the leading term from the first summand in $L_k(u)$ is $-k/q$, and in the second summand it is $k/q$ (coming from $i = q$). The two leading terms cancel, and so $L_k(u)$ is holomorphic on the disc $\{ \abs{u} < k^{-1/q}/2 \}$ with value $L_k(0) = 0$. It follows that $\exp(L_k(u))$ is also holomorphic on the same disc with value $1$ at $0$, and bounded there by $O_P(k)$. By Lemma \ref{lemma:holomorphic_FE}, the sequence $(\exp(L_k(n^{-1/q})))_{n \geq 0} = (n^{k/q} 
\Pref_{n+k}/\Pref_{n})_{n \geq 0}$ thus admits
\[
\FE_q \left(
    O_P(k),
    O_P(1),
    O_P(1),
    O_P(k^{1/q}),
    O_P(1) 
    \right). 
\]
Putting the two fractional expansions together, it follows from the product lemma that the sequence $(n^{k/q} \alpha_{n+k}/\alpha_n)_{n \geq 0}$ admits
\[
    \FE_q \left(
    O_P(k), \,
    O_P(1), \,
    O_P(1), \,
    O_P(k), \,
    O_P(1)
    \right).
\]
This completes the proof.
\end{proof}

\subsection{Asymptotics for the number of group actions}

For our application, we will need sequences $\alpha_n$ corresponding to the following polynomials. For a finite group $G$ of order $m$, let
\[
P_G(z) = \sum_{d \mid m} \frac{s_G(d)}{d} z^d,
\]
where $s_G(d)$ is the number of subgroups of $G$ of index $d$. Such polynomials satisfy \cref{eq:polynomial_coefficient_conditions}. By \cite[formula (5)]{muller1997finite}, we have
\[
    \exp (P_G(z)) = \sum_{n=0}^\infty \alpha_n(G) z^n, \quad
    \alpha_n(G) = \frac{\abs{\hom(G,\Sym(n))}}{n!}.
\]
In the particular case when $G = 1$, the coefficients of the expansion are $(1/n!)_{n \geq 0}$. Their $\FE_1$ property is Stirling's formula. Suppose now that $G$ is any finite group, and let
\[\chi_n(G) = \abs{\hom(G,\Sym(n))} = n! \, \alpha_n(G) .\]
By \Cref{thm:alphaSeq}, for every fixed $k \geq 0$, the sequence
\begin{equation}
    \label{eq:alpha_ratio_FE}
    (n^{k/q} \alpha_{n+k}(G)/\alpha_n(G))_{n \geq 0}
\end{equation}
admits a fractional expansion. Let us now convert this into a fractional expansion for the analogous sequence corresponding to $\chi_n(G)$. 

\begin{theorem}[quantitative form of \cite{PuderZimhoni2024}, Proposition~3.4]
    \label{th:chi_admits_FE}
For any $k \geq 1$, the sequence $(n^{-k/q} (n)_k \chi_{n-k}(G)/\chi_n(G))_{n \geq k}$ admits
\[
\FE_q \left(
    O_P(k^{O_P(1)}), \,
    O_P(1), \,
    O_P(1), \,
    O_P(k^2), \,
    O_P(1)
    \right).
\]
\end{theorem}
\begin{proof}
We can express the given sequence in terms of $\alpha_n(G)$ as
\begin{equation}
    \label{eq:chi_as_product}
    n^{-k/q} (n)_k \frac{\chi_{n-k}(G)}{\chi_n(G)} =
    n^{-k/q} \frac{\alpha_{n-k}(G)}{\alpha_n(G)} = \left(\frac{n-k}{n}\right)^{k/q} \cdot \frac{\alpha_{n-k}(G)}{(n-k)^{k/q} \alpha_n(G)}.
\end{equation}
For the first factor, write
\[
    \left( \frac{n-k}{n} \right)^{k/q} = (1 - k u^q)^{k/q}
\]
with $u = n^{-1/q}$. The function $(1 - k u^q)^{k/q}$ is holomorphic on the disc $\{ \abs{u} < k^{-1/q}/2 \}$ and bounded there by $2$ with value $1$ at $0$. Therefore, by Lemma \ref{lemma:holomorphic_FE}, the corresponding sequence for $n \geq k$ admits
\[
\FE_q \left(
    O_P(k), \,
    O(1), \,
    O(1), \,
    O(k^{1/q}), \,
    O(1)
    \right).
\]
For the second factor in \cref{eq:chi_as_product}, we use \Cref{thm:alphaSeq}. Shift the sequence \cref{eq:alpha_ratio_FE} by $-k$ and then take its reciprocal. Hence, the sequence
\begin{equation*}
    \label{eq:chi_1}
 \left( \frac{\alpha_{n-k}(G)}{(n-k)^{k/q} \alpha_{n}(G)} \right)_{n \geq k+1}
\end{equation*}
admits 
\[
\FE_q \left(
    O_P(k^{O_P(1)}), \,
    O_P(1), \,
    O_P(1), \,
    O_P(k^2), \,
    O_P(1)
    \right).
\]
By multiplying the two sequences above, it follows from the product lemma that the sequence given by \cref{eq:chi_as_product} admits
\[
\FE_q \left(
    O_P(k^{O_P(1)}), \,
    O_P(1), \,
    O_P(1), \,
    O_P(k^2), \,
    O_P(1)
    \right).
\]
The proof is complete.
\end{proof}

\section{Asymptotic expansion of expected traces}
\label{sec:expected_traces}

In this section, we explain how asymptotic expansions for exponentials of polynomials imply asymptotic expansions for expected traces of random representations, thus proving Assumption\,\ref{requirement:asymptotic_expansion} when $\Gamma$ is a free product of finitely many finite groups. The relationship between these two asymptotic expansions was explored by Puder and Zimhoni in \cite{PuderZimhoni2024}, where it is framed in the context of covering spaces. Here, we provide an overview of their relevant definitions and results, upgrading them for our purposes with the quantitative details.

\subsection{Subcovers and lifts}
Let $\Gamma = G_1 * G_2 * \cdots * G_m$ be as in \cref{eq:Gamma_free_product_of_Gi}. For each $1 \leq i \leq m$, let $X_{G_i}$ be a CW complex with a single vertex $v_i$ and edges $e_{i}^{x}$ indexed by nontrivial elements $x \in G_i$, such that $\pi_1(X_{G_i}, v_i) \cong G_i$. We construct $X_\Gamma$ by taking the disjoint union of $X_{G_1}, \dots, X_{G_m}$, adding a basepoint $o$, and connecting $o$ to each $v_i$ via an edge $e_i$. Thus, $\pi_1(X_{\Gamma},o) \cong \Gamma$.

A key concept from \cite{PuderZimhoni2024} is that of a \emph{subcover} of $X_\Gamma$. This is a CW complex $Y$ that embeds into some covering space of $X_\Gamma$. Each subcover comes equipped with a restricted covering map $p \colon Y \to X_\Gamma$. Subcovers can be compared and composed via morphisms that are compatible with these covering maps. A \emph{resolution} of a subcover $Y$ is a collection of morphisms of subcovers $\{ f \colon Y \to Z_f \}$ so that every morphism $h \colon Y \to \hat X$ from $Y$ to any covering space $\hat X$ of $X_\Gamma$ factors uniquely as $h = g \circ f$ for some $f$ in the resolution and some embedding $g$. Similarly, an \emph{embedding resolution} of $Y$ is a collection of injective morphisms so that every injective morphism $h$ from $Y$ to any covering space factors uniquely as $h = g \circ f$ for some $f$ in the embedding resolution and some embedding $g$.

Let $N \geq 1$ and let $\phi \colon \Gamma \to \Sym(N)$ be a uniformly random homomorphism. There exists a corresponding $N$-sheeted covering space $\pi \colon X_\phi \to X_\Gamma$ (see~\cite{hatcher2002algebraic}). Now let $p \colon Y \to X_\Gamma$ be a subcover of any covering space of $X_\Gamma$. A \emph{lift} of $p$ to $X_\phi$ is a morphism of subcovers $\tilde p \colon Y \to X_\phi$ such that $p = \pi \circ \tilde p$. Let $\E_Y(N)$ be the expected number of lifts of $p$ to $X_\phi$, that is,
\[
\E_Y(N) = \E_{\phi} [ | \{ \tilde p \colon Y \to X_\phi \mid p = \pi \circ \tilde p \} | ].
\]
Similarly, let $\E_Y^\emb(N)$ be the expected number of lifts of $p$ to $X_\phi$ that are embeddings. For a compact subcover $Y$, any \emph{finite} resolution $\mathcal R$ of $Y$ satisfies (see \cite[Lemma~2.4]{PuderZimhoni2024})
\begin{equation*}
    \E_Y(N) = \sum_{f \in \mathcal R} \E_{Z_f}^\emb(N),
\end{equation*}
and similarly for a finite embedding resolution $\mathcal R^\emb$ of $Y$,
\begin{equation}
    \label{eq:E_Y_as_sum_over_resolution}
    \E_Y^\emb(N) = \sum_{f \in \mathcal R^\emb} \E_{Z_f}^\emb(N). 
\end{equation}
Our aim now is to obtain asymptotic expansions for $\E_Y(N)$ with the help of the more accessible asymptotic expansions for $\E_Y^\emb(N)$.

\subsection{Expected number of embedding lifts for a finite group}
Suppose, in this subsection, that $\Gamma = G$ is a finite group. Let $Y$ be a compact subcover of $X_G$. Then, by \cite[Proposition~3.3]{PuderZimhoni2024}, the set
\[
\mathcal R_Y = \left\{ f \colon Y \to Z_f \, \left|  \,
\substack{\text{\normalsize $Z_f$ is a covering space of $X_G$ and} \\ 
\text{\normalsize $f(Y)$ meets every connected component of $Z_f$}} \right. 
\right\}
\]
is a finite resolution of $Y$. Similarly, define $\mathcal R_Y^\emb$ by restricting $\mathcal R_Y$ to embeddings $f$. For our purposes, we need to quantify the size of this resolution.

\begin{lemma}
\label{lem:resolution_size_finite_group}
Let $G$ be a finite group, and let $Y$ be a compact subcover of $X_G$ with $V$ vertices. Then
\[
|\mathcal R_Y| \leq 
O_{|G|}(V)^{O(V)}.
\]
\end{lemma}
\begin{proof}
The number of components of $Y$ is at most $V$. There are $O_{|G|}(1)$ isomorphism types of connected covers of $X_G$, hence the number of covers $Z$ of $X_G$ with at most $V$ components is $O_{|G|}(1)^V$. Fix one such cover $Z$. The condition that $f(Y)$ meets every component of $Z$ forces a surjection
$\pi_0(Y) \to \pi_0(Z)$, giving at most $V^V$ choices. Given the target component for each component of $Y$, a cellular map $Y \to Z$ is uniquely determined on the $1$-skeleton $Y^{(1)}$. Since $Z$ has at most $V$ components and each component is a connected at-most-$|G|$-sheeted cover of $X_G$, we must have $|V(Z)| \leq |G| V^2$. Once the images of the vertices are chosen, we have at most $O_{|G|}(1)$ further choices for each of the $O_{|G|}(V)$ edges in $Y$, since each vertex of $Z$ has $O_{|G|}(1)$ incident edges. Thus, the number of cellular maps $Y \to Z$ is at most
\[
V^V \cdot |V(Z)|^V \cdot O_{|G|}(1)^{O_{|G|}(V)}
\leq
O_{|G|}(V)^{O_{|G|}(V)}.
\]
Multiplying by the number of choices of $Z$ yields the desired bound.
\end{proof}

In the remainder of this section, we obtain quantitative bounds on the expected number of embedding lifts of subcovers of the CW complex $X_G$. This is the crucial entry point for the asymptotic expansions developed in the previous section. We begin with the case where $Y$ is a compact covering space.

\begin{proposition}[quantitative form of \cite{PuderZimhoni2024}, Proposition~3.4] \label{prop:embedding_lifts_finite_group_cover}
Let $G$ be a finite group and let $Y$ be a compact covering space of $X_G$ with $V$ vertices. Then the sequence $(N^{-V/|G|} \E^\emb_Y(N))_{N \geq V}$ admits
\[
    \FE_{|G|} \left(
        O_{|G|}(V^{O_{|G|}(1)}), \,
        O_{|G|}(1), \,
        O_{|G|}(1), \,
        O_{|G|}(V^2), \,
        O_{|G|}(1)
    \right) .
\]
\end{proposition}
\begin{proof}
The number of $N$-sheeted covers of $X_G$ is precisely $\abs{\hom(G, \Sym(N))} = \chi_N(G)$ (see \cite{hatcher2002algebraic}). In every embedding $h \colon Y \to \hat{X}$ of $Y$ into an $N$-sheeted cover of $X_G$, the image $h(Y)$ contains $V$ vertices of $\hat{X}$. Since $Y$ is a covering space, $h(Y)$ and its complement are disconnected. The embedding $h$ is thus determined by first selecting the $V$ distinct vertices among $N$ of them, and then selecting a complementing $(N-V)$-sheeted cover. In total,
\[
\E_Y^\emb(N) =
(N)_V \cdot \frac{\chi_{N-V}(G)}{\chi_N(G)}.
\]
The statement now follows immediately from \Cref{th:chi_admits_FE}.
\end{proof}

We are now ready to handle any compact subcover $Y$ of $X_G$. Following \cite[Definition 2.5]{PuderZimhoni2024}, we associate the Euler characteristic $\chi^\grp(Y)$ to $Y$. This characteristic is defined as the sum of the Euler characteristics of the groups\footnote{For an overview of Euler characteristics of groups, see \cite[Definition 1.3]{PuderZimhoni2024}.} $p_*(\pi_1(Y_i))$, where $Y_i$ ranges over the connected components of $Y$ and $p \colon Y \to X_G$ is the restricted covering map. For a finite group $G$, we always have $\chi^\grp(Y) \geq 0$. Moreover, if $Y$ is a cover of $X_G$, then we have $\chi^\grp(Y) = |V(Y)|/|G|$ (see the final part of the proof of Proposition 3.4 in \cite{PuderZimhoni2024}). 

\begin{proposition}[quantitative form of \cite{PuderZimhoni2024}, Corollary 3.5]
Let $G$ be a finite group and let $Y$ be a compact subcover of $X_G$ with $V$ vertices. Then the sequence $(N^{-\chi^{\grp}(Y)} \E_Y^\emb(N))_{N \geq V}$ admits
\[
    \FE_{|G|} \left(
        O_{|G|}(V^{O_{|G|}(1)}), \,
        O_{|G|}(1), \,
        O_{|G|}(V)^{O_{|G|}(V)}, \,
        O_{|G|}(V^2), \,
        O_{|G|}(1)
    \right) .
\]
\end{proposition}
\begin{proof}
Expand $\E_Y^\emb(N)$ using \cref{eq:E_Y_as_sum_over_resolution} over the resolution $\mathcal R_Y^\emb$ as
\[
\E_Y^\emb(N) = 
\sum_{f \in \mathcal R_Y^\emb} \E_{Z_f}^\emb(N).
\]
Each summand $\E_{Z_f}^\emb(N)$ corresponds to a compact cover $Z_f$ of $X_G$ with at most $V$ components. In particular, $Z_f$ has $V_f = O_{|G|}(V)$ vertices. It now follows from the previous proposition that each sequence
\[(N^{-V_f/|G|} \E_{Z_f}^\emb(N))_{N \geq V}\]
admits
\[
    \FE_{|G|} \left(
        O_{|G|}(V^{O_{|G|}(1)}), \,
        O_{|G|}(1), \,
        O_{|G|}(1), \,
        O_{|G|}(V^2), \,
        O_{|G|}(1)
    \right) .
\]
At the same time, the size of the embedding resolution is bounded from \Cref{lem:resolution_size_finite_group} by $\abs{\mathcal R_Y} \leq O_{|G|}(V)^{O(V)}$. Lemma~\ref{lemma:clSum} implies that, for some constant $T \geq 1$ the sequence
\[
    \left( 
        \frac{N^{-\max_{f} V_f/|G|}}{T} \E_Y^\emb(N) 
    \right)_{N \geq V}
\]
admits
\[
    \FE_{|G|} \left(
        O_{|G|}(V^{O_{|G|}(1)}), \,
        O_{|G|}(1), \,
        O_{|G|}(V)^{O_{|G|}(V)}, \,
        O_{|G|}(V^2), \,
        O_{|G|}(1)
    \right).
\]
Now, by \cite[Proof of Corollary 3.5]{PuderZimhoni2024}, the maximum $\max_f V_f/|G| = \max_f \chi^\grp(Z_f)$ is attained at a unique $f \in \mathcal R_Y^\emb$, and its value there is $\chi^\grp(Y)$. Hence $T = 1$ and the proof is complete.
\end{proof}

\subsection{Expected number of embedding lifts for a free product}

Suppose now that $\Gamma = G_1 * G_2 * \cdots * G_m$ is a free product of finite groups as in \cref{eq:Gamma_free_product_of_Gi}. Set
    \begin{equation}
        \label{eq:mu_and_M_for_Gamma}
        \mu(\Gamma)=\lcm \bigl(|G_1|,\dots,|G_m|\bigr),
        \qquad
        M(\Gamma)=|G_1|+\cdots+|G_m|.
    \end{equation}
Using the results of the previous subsection, we can now prove fractional expansions for expected number of lifts for any compact subcover of $X_\Gamma$.

\begin{proposition}
    \label{prop:embFinite}
    Let $\Gamma = G_1 * G_2 * \cdots * G_m$ with $\mu = \mu(\Gamma)$ and $M = M(\Gamma)$ as in \cref{eq:mu_and_M_for_Gamma}, and let $Y$ be a compact subcover of $X_\Gamma$ with $V$ vertices. Then the sequence $(N^{-\chi^\grp(Y)} \E_Y^\emb(N))_{N \geq V}$ admits
    \[
\FE_\mu \left(
    O_{\mu,M}(V^{O_{\mu,M}(1)}), \,
    O_{M}(1), \,
    O_{M}(V)^{O_{M}(V)}, \,
    O_{\mu,M}(V^2), \,
    O_M(1)
    \right).
    \]
\end{proposition}
\begin{proof}
Let $p \colon Y \to X_\Gamma$ be the restricted covering map. For $1 \leq i \leq m$,  let $Y |_{G_i} = p^{-1}(X_{G_i})$ be the subcomplex of $Y$ lying above $X_{G_i}$, and let $v = |p^{-1}(o)|$, $v_i = |V(Y|_{G_i})|$, $\epsilon_i = |p^{-1}(e_i)|$. All these numbers are at most $V$. It is shown in \cite[proof of Proposition 2.6]{PuderZimhoni2024} that
\[
    \E_Y^\emb(N) =
    (N)_v \cdot
    \prod_{i=1}^m
    \left( \frac{1}{(N)_{\epsilon_i}} \right) 
    \cdot
    \left( 
    \prod_{i=1}^m
    \E_{Y|_{G_i}}^\emb(N)
    \right).
\]
Recall from Example \ref{example:pochhammer_FE} that the sequences $(N^v/(N)_v)_{N \geq v}$ and $(N^{\epsilon_i}/(N)_{\epsilon_i})_{N \geq \epsilon_i}$ admit
\[
\FE_\mu \left(
    O_\mu(V^2), \, 
    1, \, 
    O(1), \, 
    O_\mu(V^2), \, 
    1
    \right).
\]
Inverting the first sequence into $((N)_v / N^v)_{N \geq v}$ gives, by the reciprocal lemma, a fractional expansion of the form
\[
\FE_\mu \left( 
    O_\mu(V^{O_{\mu}(1)}), \,
    1, \,
    O(1), \,
    O_\mu(V^2), \,
    O(1)
\right).
\]
By the previous proposition, each sequence
\[(N^{-\chi^\grp(Y|_{G_i})} \E_{Y|_{G_i}}^\emb(N))_{N \geq V}\]
admits
\[
    \FE_\mu \left(
        O_{M}(V^{O_{M}(1)}), \,
        O_{M}(1), \,
        O_{M}(V)^{O_{M}(V)}, \,
        O_{M}(V^2), \,
        O_{M}(1)
    \right).
\]
It remains to multiply all these fractional expansions together. The product lemma yields that the sequence 
\[
    \left(
        N^{ - v + \sum_{i=1}^m \epsilon_i - \sum_{i=1}^m \chi^\grp(Y|_{G_i}) } \E_Y^\emb(N)
    \right)_{N \geq V}
\]
admits 
\[
\FE_\mu \left(
    O_{\mu,M}(V^{O_{\mu,M}(1)}), \,
    O_{M}(1), \,
    O_{M}(V)^{O_{M}(V)}, \,
    O_{\mu,M}(V^2), \,
    O_M(1)
    \right).
\]
The exponent $- v + \sum_{i=1}^m \epsilon_i - \sum_{i=1}^m \chi^\grp(Y|_{G_i})$ is equal to $\chi^\grp(Y)$ by \cite[Lemma 4.2]{PuderZimhoni2024}, completing the proof.
\end{proof}

\subsection{Expected traces of random representations}

Finally, we are in a position to prove \Cref{requirement:asymptotic_expansion} for the case where $\Gamma = G_1 * G_2 * \cdots * G_m$ as in \cref{eq:mu_and_M_for_Gamma}. In order to set up our notation, let $\gamma$ be a nontrivial element of $\Gamma$, written in normal form as
\[
    \gamma = x_1 x_2 \cdots x_{|\gamma|}, \qquad\hbox{for some } x_i \in G_{i_j} \backslash \{ 1 \}, \, i_{j+1} \neq i_j.
\]
To each $x_i$ there corresponds an edge $e_{i_j}^{x_i}$ in the complex $X_{G_{i_j}}$. The element $\gamma \in \pi_1(X_\Gamma)$ can be represented as the loop traversing, in order, for each $i = 1, 2, \dots, \abs{\gamma}$, the segments
\[
    o \overset{e_{i_j}}{\longrightarrow} v_{i_j} 
    \overset{e_{i_j}^{x_i}}{\longrightarrow} v_{i_j}
    \overset{e_{i_j}}{\longrightarrow} o.
\]
Let $C_\gamma$ be the cycle with $3\abs{\gamma}$ vertices mapping onto the combinatorial representation of $\gamma$ above. By the lifting criterion in covering spaces (see, for example, \cite[Proposition 2.8]{PuderZimhoni2024}), this map can be lifted to the universal cover of $X_\Gamma$. Let $Y_\gamma$ be the image of this lift, equipped with the restricted covering map $p \colon Y_\gamma \to X_\Gamma$. The subcover $Y_\gamma$ is connected and has at most $3 \abs{\gamma}$ vertices and edges.

Suppose now that $\phi \colon \Gamma \to \Sym(N)$ is a uniformly random representation, and let $X_\phi$ be the corresponding $N$-sheeted covering space of $X_\Gamma$. The number of fixed points $\fix \phi(\gamma)$ of the action of $\gamma$ on $\{1, 2, \dots, N\}$ is, again by the lifting criterion, equal to the number of lifts of $p \colon Y_\gamma \to X_\Gamma$ to $X_\phi$ (see \cite{hatcher2002algebraic}). We can thus express the expected trace of $\phi(\gamma)$ as
\begin{equation}
    \label{eq:expected_trace_as_expected_lifts}
    \E_{\phi}[\ntr \phi(\gamma)] = 
    N^{-1} \E_{\phi}\left[ \abs{\fix \phi(\gamma)} \right] = 
    N^{-1} \E_{Y_\gamma}(N).
\end{equation}
We can expand the last expectation using the resolution
\[
\mathcal R_\gamma = \left\{ 
    f \colon Y_\gamma \to Z_f \mid \text{$f$ is a surjective morphism of subcovers}
\right\}.
\]

\begin{lemma}
\label{lemma:resolutionSize}
For every nontrivial $\gamma \in \Gamma$,
\[
\abs{\mathcal R_\gamma} \leq O(\abs{\gamma})^{O(|\gamma|)}.
\]
\end{lemma}
\begin{proof}
The number of vertices and edges of $Y_\gamma$ is at most $3\abs{\gamma}$. Any morphism of subcovers $f \colon Y_\gamma \to Z$ is determined by its action on the $1$-skeleton. We thus count the number of surjective morphisms by giving a crude upper bound on the number of ways to specify the images of vertices and edges.  Because $f$ is surjective on vertices, the map $f_V \colon V(Y_\gamma)\to V(Z)$ induces a partition $\Pi$ of $V(Y_\gamma)$ into nonempty fibers, with each block corresponding to a vertex of $Z$. Hence the number of possibilities for $\Pi$ is at most the Bell number $B_{|V(Y_\gamma)|} \leq (3 \abs{\gamma})^{3 |\gamma|}$.

Now fix one such partition $\Pi$. For each pair of blocks $\{A,B\}$ in $\Pi$, let $E_{AB}\subseteq E(Y_\gamma)$ be the set of edges of $Y_\gamma$ whose endpoints lie in $A$ and $B$. Since $f$ preserves incidence, every edge in $E_{AB}$ must map to an edge of $Z_f$ whose endpoints are the vertices corresponding to $A$ and $B$. Thus, for each fixed pair $\{A,B\}$, specifying $f_E \colon E(Y_\gamma)\to E(Z)$ on $E_{AB}$ is equivalent to partitioning $E_{AB}$ into nonempty fibers (each fiber being the preimage of a single edge of $Z$ between $A$ and $B$). Therefore the number of possibilities for $f_E$ is at most
\[
\prod_{\substack{\{A,B\} \in \Pi \\ E_{AB} \neq \emptyset}} \abs{E_{AB}}^{|E_{AB}|}.
\]
Using $\sum_{\{A,B\} \in \Pi} \abs{E_{AB}} = \abs{E(Y_\gamma)} \leq 3 \abs{\gamma}$ and the inequality $\prod_i a_i^{a_i}\le (\sum_i a_i)^{\sum_i a_i}$ for $a_i > 0$, it follows that the above is at most $(3 \abs{\gamma})^{3 |\gamma|}$.

Putting things together, the total number of surjective morphisms $f \colon Y_\gamma\to Z_f$ is
\[
\abs{\mathcal R_\gamma}
\leq
(3 \abs{\gamma})^{6 |\gamma|},
\]
as desired.
\end{proof}

We can now derive the desired asymptotic expansion for expected traces. In order to state it, we need to introduce some notation. Let $\Gamma^\tors$ denote the set of torsion elements of $\Gamma$, and let $\Gamma^\tf = \Gamma \setminus \Gamma^\tors$ denote the set of elements of infinite order. We define a parameter depending on a group element $\gamma \in \Gamma$ as
\[h(\gamma) =
    \begin{cases}
        1 & \gamma \in \Gamma^\tors \\
        \abs{\mathcal H_\gamma} & \gamma \in \Gamma^\tf ,
    \end{cases}
\]
where $\mathcal H_\gamma$ is the set of subgroups of $\Gamma$ containing $\gamma$ that are isomorphic to either $\Z$ or $C_2 * C_2$. (The value of $h(\gamma)$ will be computed explicitly in Lemmas~\ref{lem:leading_coefficients_for_group_elements} and~\ref{lem:counting_C2C2_and_powers}.) Additionally, for $\gamma \in \Gamma^\tf$, we interpret $1/|\langle \gamma \rangle|$ as $0$.

\begin{theorem}
    \label{thm:embFree}
Let $\Gamma = G_1 * G_2 * \cdots * G_m$ with $\mu = \mu(\Gamma)$ and $M = M(\Gamma)$ as in \cref{eq:mu_and_M_for_Gamma}. For every $\gamma \in \Gamma$, the sequence
\[
    \left(
        \frac{N^{-1/|\langle \gamma \rangle|}}{h(\gamma)} \cdot \E_\phi \left[ 
            \tr \phi_N(\gamma)
        \right]
    \right)_{N \gg_M |\gamma|}
\]
admits
\[
\FE_\mu \left(
    O_{\mu,M}(\abs{\gamma})^{O_{\mu,M}(1)}, \,
    O_{M}(1), \,
    O_{M}(\abs{\gamma})^{O_{M}(|\gamma|)}, \,
    O_{\mu,M}(\abs{\gamma}^2), \,
    O_M(1)
    \right).
\]
\end{theorem}
\begin{proof}
Suppose first that $\gamma$ is not a torsion element. We can expand \cref{eq:expected_trace_as_expected_lifts} using the resolution $\mathcal R_\gamma$. Each summand $\E_{Z_f}^\emb(N)$ corresponds to a compact subcover $Z_f$ of $X_\Gamma$ with at most $3 \abs{\gamma}$ vertices. It now follows from Proposition~\ref{prop:embFinite} that each sequence $(N^{-\chi^\grp(Z_f) /\mu} \E_{Z_f}^\emb(N))_{N \geq 3 |\gamma|}$ admits
\[
\FE_\mu \left(
    O_{\mu,M}(\abs{\gamma})^{O_{\mu,M}(1)}, \,
    O_{M}(1), \,
    O_{M}(\abs{\gamma})^{O_{M}(|\gamma|)}, \,
    O_{\mu,M}(\abs{\gamma}^2), \,
    O_M(1)
    \right).
\]
Let $e = \max_f \chi^\grp(Z_f)$. By Lemma~\ref{lemma:resolutionSize}, the size of the resolution is bounded by $O(\abs{\gamma})^{O(|\gamma|)}$. The sum lemma now implies that the sequence $(T^{-1} N^{-e/\mu} \E_Y(N))_{N \geq 3 |\gamma|}$ admits
\[
\FE_\mu \left(
    O_{\mu,M}(\abs{\gamma})^{O_{\mu,M}(1)}, \,
    O_{M}(1), \,
    O_{M}(\abs{\gamma})^{O_{M}(|\gamma|)}, \,
    O_{\mu,M}(\abs{\gamma}^2), \,
    O_M(1)
    \right).
\]
The values of $e$ and $T$ can be identified as follows. It is shown in \cite[proof of Theorem 1.4]{PuderZimhoni2024} that $\chi^\grp(Y_\gamma) \leq 0$ for all $f$, and the maximum value of $0$ is attained at least once. In particular, we have $e = 0$, and $T$ corresponds to the number of $f$ for which $\chi^\grp(Z_f) = 0$. The number of such $f$ is, as stated in \cite[Theorem 1.4]{PuderZimhoni2024}, equal to the number of subgroups of $\Gamma$ containing $\gamma$ that are isomorphic to $\Z$ or $C_2 * C_2$, which we denoted above as $h(\gamma)$.

On the other hand, when $\gamma$ is a torsion element, we can proceed in a more straightforward manner and use a simpler $Y_\gamma$. In this case, all the statistics about $\pi_N(\gamma)$ depend only on the finite factor $G$ of $\Gamma$ into which $\gamma$ conjugates. In particular, let $Y_\gamma$ be the finite cover of $X_G$ corresponding to the finite cyclic subgroup generated by $\gamma$ inside $\pi_1(X_G)$. The number of vertices of $Y_\gamma$ is the index $|G : \langle \gamma \rangle|$. In this case, we can use the more direct fractional expansion from Proposition \ref{prop:embedding_lifts_finite_group_cover} to deduce that the sequence
\[
    \left( N^{-1/|\langle \gamma \rangle|} \E_{Y_\gamma}(N) \right)_{N \geq |G:\langle \gamma \rangle|}
\]
admits $\FE_{|G|}$, and hence $\FE_{\mu}$, with all parameters bounded by $O_{|G|}(1) = O_{M}(1)$.
\end{proof}

\subsection{Asymptotic expansion of expected traces}

Finally, we are ready to deduce \Cref{requirement:asymptotic_expansion}.

\begin{theorem}
    \label{th:proof_of_assumption_1}
Let $\Gamma = G_1 * G_2 * \cdots * G_m$ with $\mu = \mu(\Gamma)$ and $M = M(\Gamma)$ as in \cref{eq:mu_and_M_for_Gamma}. For every $\gamma \in \Gamma$, every $s \geq \abs{\gamma}$, and every $N \geq O_{\mu,M}(s)^{O_{\mu,M}(1)}$, there exist constants $a_\ell \in \R$ for all non-negative integers $\ell$ such that $a_0 = \1_{\gamma = 1} = \tau(\lambda(\gamma))$ and
\[
\abs{\E \left[ \ntr \pi_N(\gamma) \right]
-
\sum_{\ell = 0}^{s-1} a_\ell N^{-\ell/\mu}}
\leq
O_{\mu,M}(s)^{O_{\mu,M}(s)} N^{-s/\mu}.
\]
\end{theorem}
\begin{proof}
For every $\gamma \in \Gamma$, and for a random homomorphism $\phi \colon \Gamma \to \Sym(N)$ with its extension $\pi_N = \std \circ \phi_N$, we have
\begin{equation}
    \label{eq:expected_trace_pi_to_Trace_phi}
    \E \left[ \ntr \pi_N(\gamma) \right] = 
    N^{-1 + 1/|\langle \gamma \rangle|} \left( N^{-1/|\langle \gamma \rangle|} \E \left[ \tr \phi_N(\gamma) \right] - N^{-1/|\langle \gamma \rangle|} \right).
\end{equation}
By \Cref{thm:embFree}, there are coefficients $b_k \in \R$ depending on $\gamma$ such that, for all $s \geq 1$ and all $N \geq O_{\mu,M}(\abs{\gamma})^{O_{\mu,M}(1)} s^{O_M(1)}$, we have
\[
N^{-1/|\langle \gamma \rangle|} \E \left[ \tr \phi_N(\gamma) \right]
=
\sum_{k=0}^{s-1} b_k N^{-k/\mu} + X
\]
with $b_0 = h(\gamma)$ and
\[
\abs{X} \leq 
O_M(\abs{\gamma})^{O_M(|\gamma|)} \cdot O_{\mu,M}(|\gamma|^2 s)^{O_M(1)s} N^{-s/\mu}.
\]
Moreover, by Lemma~\ref{lemma:coeffBound}, we have
\[
\abs{b_k} \leq 
O_M(\abs{\gamma})^{O_M(|\gamma|)} \cdot O_{\mu,M}(|\gamma|^2 k)^{O_M(1) k}.
\]
Inserting this expansion into \cref{eq:expected_trace_pi_to_Trace_phi} gives
\[
\E \left[ \ntr \pi_N(\gamma) \right]
=
\sum_{k = 0}^{s-1} b_k N^{-1 + 1/|\langle \gamma \rangle| - k/\mu} - N^{- 1} + N^{-1 + 1/|\langle \gamma \rangle|} X.
\]
Set $t = \mu (1 - 1/|\langle \gamma \rangle|) \geq 0$. Note that $t \leq \mu$ with equality only when $\gamma$ is of infinite order. Set furthermore $\ell = t + k$ as $k$ varies. Then the above can be rewritten as
\begin{equation}
    \label{eq:expected_trace_ap_plus_error}
    \sum_{\ell = t}^{t + s - 1} a_\ell N^{-\ell/\mu}
    + N^{-1 + 1/|\langle \gamma \rangle|} X,
    \qquad
    a_\ell = b_{\ell - t} - \1_{\ell = \mu}.
\end{equation}
Split the sum above into the main part with $\ell$ ranging up to $s-1$, and into the tail with $\ell$ from $s$ to $t + s - 1$. The latter part is then bounded by
\begin{equation}
\label{eq:previousBound}
\sum_{\ell = s}^{s + t - 1} |a_\ell| N^{-\ell/\mu} \leq
t \cdot O_M(\abs{\gamma})^{O_M(|\gamma|)} \cdot O_{\mu,M}(|\gamma|^2 s)^{O_M(1)s} \cdot N^{-s/\mu}.
\end{equation}
Suppose that $s \geq \abs{\gamma}$. Thus, for $N \geq O_{\mu,M}(s)^{O_{\mu, M}(1)}$, we have
\[
N^{-1 + 1/|\langle \gamma \rangle|} \abs{X} \leq O_{\mu,M}(s)^{O_{\mu,M}(s)} N^{-s/\mu}.
\]
Combining this with \cref{eq:previousBound}, we obtain from \cref{eq:expected_trace_ap_plus_error} that
\[
\abs{\E \left[ \ntr \pi_N(\gamma) \right]
-
\sum_{\ell = 0}^{s-1} a_\ell N^{-\ell/\mu}}
\leq
O_{\mu,M}(s)^{O_{\mu,M}(s)} N^{-s/\mu}.
\]
This is the desired expansion. Note that the coefficients $a_\ell$ vanish for all $\ell$ satisfying $0 \leq \ell < t = \mu(1 - 1/|\langle \gamma \rangle|)$. When $\gamma$ is of infinite order, this means that $a_\ell = 0$ for $0 \leq \ell < \mu$, while $a_\mu = a_t = b_0 - 1 = \abs{\mathcal H_\gamma} - 1$ by \cref{eq:expected_trace_ap_plus_error}. On the other hand, when $\gamma$ is torsion, $a_0$ can only be nonzero when $t = 0$, which is equivalent to $\gamma = 1$ in $\Gamma$. In this case, $a_0 = b_0 = 1$. Overall, we thus have $a_0 = \1_{\gamma = 1} = \tau(\lambda(\gamma))$, as required.
\end{proof}

\section{Temperedness of leading coefficients}
\label{sec:temperedness}

In this section, we prove that temperedness of leading coefficients, that is Requirement~\ref{requirement:temperedness}, holds when $\Gamma$ is a free product of finitely many finite groups.

\subsection{Positivization trick}

A discrete group $\Gamma$ equipped with a generating set $S$ satisfies the \emph{rapid decay property} if there exists a polynomial $P \in \C[X]$ such that for every $x \in \C[\Gamma]$ we have
\[ 
\norm{\lambda(x)}
\leq P(|x|) \cdot
\norm{x}_{\ell^2(\Gamma)},
\]
where $\norm{\cdot}$ is the operator norm on $\bounded(\ell^2(\Gamma))$. This notion has been originally proved to hold for the free group by Haagerup \cite{Haagerup1979}, while it was formally introduced later by Jolissaint \cite{Jolissaint1990}. The rapid decay property is stable under taking free products \cite[Theorem A]{Jolissaint1990}, and it holds for finite groups \cite[Proposition B]{Jolissaint1990}. In particular, the groups $\Gamma$ we are considering here satisfy the rapid decay property.

Under the rapid decay property, Magee and de la Salle \cite[Section 6.2]{MageeSalle2024} develop a reduction to proving temperedness of leading coefficients to elements $x \in \C[\Gamma]$ that are \emph{symmetric generators of random walks} on $\Gamma$, meaning that in the expansion $x = \sum_{\gamma \in \Gamma} \alpha_\gamma \gamma$, the coefficients $\alpha_\gamma$ are non-negative, sum to $1$, and $\alpha_{\gamma^{-1}} = \alpha_\gamma$. Here is the precise claim, known in the literature as the \emph{positivization trick} (see also \cite[Section 2.6]{van2025strong}).

\begin{lemma}[\cite{MageeSalle2024}, Proposition 6.3] \label{lem:positivization}
    Let $\Gamma$ be a finitely generated group with the rapid decay property. Let $u \colon \C[\Gamma] \to \C$ be a function that is tempered at every generator of a symmetric random walk on $\Gamma$. Then $u$ is tempered at every self-adjoint element of $\C[\Gamma]$.
\end{lemma}

\subsection{Leading coefficients for group elements}

In order to apply the positivization trick, we need to understand the leading coefficients $u_k(\gamma)$ for group elements $\gamma \in \Gamma$ a bit better. 

\begin{lemma}
    \label{lem:leading_coefficients_for_group_elements}
Let $\Gamma = G_1 * G_2 * \cdots * G_m$ with $\mu = \mu(\Gamma)$ and $M = M(\Gamma)$ as in \cref{eq:mu_and_M_for_Gamma}. Let $\gamma \in \Gamma$ be any element, and let $u_k(\gamma)$ be as in \Cref{requirement:asymptotic_expansion}.
\begin{enumerate}
    \item If $\gamma \in \Gamma^\tors$, then $u_k(\gamma) \leq O_{\mu,M}(1)$.
    \item If $\gamma \in \Gamma^\tf$, then $u_k(\gamma) = 0$ for $0 \leq k < \mu$, and $u_\mu(\gamma) = \abs{\Hcal_\gamma} - 1$.
\end{enumerate}
\end{lemma}
\begin{proof}
Let $\gamma \in \Gamma^\tors$. Then $\gamma$ is conjugate to an element of a finite free factor $G_i$ of $\Gamma$. Since $u_k$ is invariant under conjugation, we can thus assume that $\gamma$ belongs to a finite group of order at most $M$. There are $O_M(1)$ such choices for $\gamma$, and there are at most $\mu$ many values $u_k(\gamma)$ as $1 \leq k \leq \mu$. This proves the first portion of the statement. Meanwhile, the second part, $\gamma \in \Gamma^\tf$, is contained in the final part of the proof of \Cref{th:proof_of_assumption_1}.
\end{proof}

We also have a complete description of $\abs{\Hcal_\gamma}$ due to the rigid structure of $\Gamma^\tf$.

\begin{lemma}
    \label{lem:counting_C2C2_and_powers}
Let $\Gamma = G_1 * G_2 * \cdots * G_m$. For any $\gamma \in \Gamma^\tf$, there is a unique $\delta \in \Gamma$ and $d \geq 1$ such that $\gamma = \delta^d$ and $\delta$ is not a proper power in $\Gamma$. Moreover,
\[
    \abs{\Hcal_\gamma} = \omega(d) + 
    \begin{cases}
        \sigma(d) & \text{if } \gamma \text{ is contained in a subgroup isomorphic to } C_2 * C_2, \\
        0 & \text{otherwise},
    \end{cases}
\]
where $\omega(d)$ is the number of positive divisors of $d$, and $\sigma(d)$ is their sum.
\end{lemma}
\begin{proof}
 Elements of infinite order in $\Gamma$ have cyclic centralizers by \cite[Problem 28, p.196]{magnus2004combinatorial}. Let $\gamma \in \Gamma^\tf$, and let $\delta$ be a generator of $C_\Gamma(\gamma)$. Then $\gamma = \delta^d$ for some $d \geq 1$, and $\delta$ is not a proper power. Inside $\langle \delta \rangle \cong \Z$, there are exactly $\omega(d)$ subgroups containing $\gamma$. Suppose now that $\gamma$ is also contained in a subgroup isomorphic to $C_2 * C_2$, generated by an involution $\tau$ and an element $\sigma$ of infinite order. Then $\gamma \in \langle \sigma \rangle$, hence $\sigma \in C_\Gamma(\gamma) = \langle \delta \rangle$. Thus $\sigma = \delta^k$ for some $k \in \Z$ dividing $d$. In other words, any $C_2 * C_2$ containing $\gamma$ is of the form $\langle \tau, \delta^k \rangle$ for some involution $\tau$ with $\tau \gamma \tau = \gamma^{-1}$ and some $k \mid d$. Suppose such an involution $\tau$ indeed exists. The involutions in $\langle \tau, \delta \rangle$ are precisely of the form $\tau \delta^i$ for some $i$. This means that all $C_2 * C_2$ subgroups containing $\gamma$ are $H_{i,k} = \langle \tau \delta^i, \delta^k \rangle$ for $0 \leq i < k$ and $k \mid d$. Moreover, all these subgroups are distinct. Thus, if there exists at least one $C_2 * C_2$ subgroup containing $\gamma$, then the number of such subgroups is precisely $\sum_{k \mid d} k = \sigma(d)$.
\end{proof}

We can thus split $\Gamma^\tf$ into three disjoint subsets:
\begin{itemize}
    \item $\Gamma_\pow^\tf$: elements $\gamma \in \Gamma^\tf$ such that $\gamma = \delta^d$ for some $\delta \in \Gamma$ that is not a proper power and some integer $d \geq 2$. In this case, $\abs{\Hcal_\gamma} \leq \omega(d) + \sigma(d)$. 
    
    \item $\Gamma_{C_2 * C_2}^\tf$: elements $\gamma \in \Gamma^\tf$ that are not proper powers and are contained in a subgroup of $\Gamma$ isomorphic to $C_2 * C_2$. In this case, $\abs{\Hcal_\gamma} = 2$.
    
    \item $\Gamma_\infty^\tf$: elements $\gamma \in \Gamma^\tf$ that are neither proper powers nor contained in a subgroup isomorphic to $C_2 * C_2$. In this case, $\abs{\Hcal_\gamma} = 1$.
\end{itemize}

\subsection{Hitting probabilities for random walks}

Let $x \in \C[\Gamma]$ be a generator of a symmetric random walk on $\Gamma$. Let $X$ be the random walk on $\Gamma$ driven by $x$, that is, $X$ is a random variable with law $\P[X=\gamma] = \alpha_\gamma$. For every $p \geq 1$, let $X_1, X_2, \dots, X_p$ be independent and identically distributed copies of $X$. Thus $X_1 X_2 \cdots X_p$ is the position of the random walk after $p$ steps. Note that
\begin{equation}
\label{eq:xp_expectation}
    x^p = 
    \E\left[ X_1 X_2 \cdots X_p \right] =
    \sum_{\gamma \in \Gamma} \P(X_1 X_2 \cdots X_p = \gamma) \gamma.
\end{equation}
We will require upper bounds on the hitting probabilities of specific subsets of $\Gamma$, consisting of elements $\gamma$ for which we have uniform control over $u_k(\gamma)$. During this, we will rewrite \cref{eq:xp_expectation} in terms of the left-regular representation. Note that for any $\gamma, \alpha \in \Gamma$, we can write $\1_{\gamma = \alpha} = \langle \delta_\gamma, \lambda(\alpha) \delta_1 \rangle$. Using this with $\alpha = X_1 X_2 \cdots X_p$ and taking expectations, we obtain from \cref{eq:xp_expectation} that
\begin{equation}
    \label{eq:hitting_probability_inner_product}
    \P(X_1 X_2 \cdots X_p = \gamma) =
    \langle \delta_\gamma , \lambda(x)^p \delta_1 \rangle.
\end{equation}

\begin{lemma}\label{lemma:rwTorsion}
Let $\Gamma = G_1 * G_2 * \cdots * G_m$ with $M = M(\Gamma)$ as in \cref{eq:mu_and_M_for_Gamma}. Then
\[
    \P(X_1 X_2 \cdots X_p \in \Gamma^\tors) \leq 
    (p+1)^2 M \cdot \norm{\lambda(x)}^p.
\]
\end{lemma}
\begin{proof}
Any torsion element $\gamma \in \Gamma^\tors$ is conjugate to an element of a finite free factor $G_i$ of $\Gamma$. Thus, there exist $\alpha, \beta \in \Gamma$ such that $\gamma = \alpha^{-1} \beta \alpha$, and $\beta$ lies in a finite free factor. Hence
\begin{equation}
    \label{eq:hitting_probability_torsion_conjugacy}
    \P(X_1 X_2 \cdots X_p \in \Gamma^\tors) \leq
    \sum_{\beta \in G_1 \cup G_2 \cup \dots \cup G_k} \sum_{\alpha \in \Gamma} \P(X_1 X_2 \cdots X_p = \alpha^{-1} \beta \alpha).
\end{equation}
The number of possible $\beta$ is bounded by the sum of $\abs{G_i}$, that is, by $M$. Let us then bound the inner sum over $\alpha$. Since the group $\Gamma$ admits normal form of elements, the event $X_1 X_2 \cdots X_p = \alpha^{-1} \beta \alpha$ implies that there exist $1 \le t_1 \le t_2 \le p+1$ such that
\[
    X_1 X_2 \cdots X_{t_1 - 1} = \alpha^{-1}, \quad
    X_{t_1} X_{t_1 + 1} \cdots X_{t_2 - 1} = \beta, \quad
    X_{t_2} X_{t_2 + 1} \cdots X_p = \alpha.
\]
The independence of the $X_i$, together with \cref{eq:hitting_probability_inner_product}, implies that the probability $\P(X_1 X_2 \cdots X_p = \alpha^{-1} \beta \alpha)$ is thus bounded by
\[
\sum_{1 \le t_1 \le t_2 \le p+1}
\langle \delta_{\alpha^{-1}}, \lambda(x)^{t_1 - 1} \delta_1 \rangle
\langle \delta_{\beta}, \lambda(x)^{t_2 - t_1} \delta_1 \rangle
\langle \delta_{\alpha}, \lambda(x)^{p + 1 - t_2} \delta_1 \rangle.
\]
The middle factor is at most $\norm{\lambda(x)}^{t_2 - t_1}$. Therefore, the inner sum over $\alpha \in \Gamma$ in \cref{eq:hitting_probability_torsion_conjugacy} is bounded by
\[
\norm{\lambda(x)}^{t_2 - t_1} 
\sum_{1 \le t_1 \le t_2 \le p+1}
\sum_{\alpha \in \Gamma}
\langle \delta_{\alpha^{-1}}, \lambda(x)^{t_1 - 1} \delta_1 \rangle
\langle \delta_{\alpha}, \lambda(x)^{p + 1 - t_2} \delta_1 \rangle.
\]
Note that for any $t \geq 1$, we have
\[ 
\sum_{\alpha \in \Gamma} 
\langle \delta_{\alpha}, \lambda \left(x\right)^{t} \delta_1 \rangle ^2 =
\norm{
    \lambda(x)^t \delta_1 }_{\ell^2(\Gamma)}^2
\leq \norm{ \lambda(x) }^{2t} . 
\]
This yields, using the Cauchy-Schwarz inequality,
\[
\sum_{\alpha \in \Gamma} 
\langle \delta_{\alpha^{-1}}, 
\lambda(x)^{t_1-1} \delta_1 \rangle
\langle	\delta_\alpha,	
\lambda(x)^{p+1-t_2} \delta_1 \rangle
\leq \norm{ \lambda(x) }^{p+t_1-t_2}.
\]
The statement now follows immediately by plugging these bounds into \cref{eq:hitting_probability_torsion_conjugacy}, taking into account that there are at most $(p+1)^2$ choices for the pair $(t_1,t_2)$.
\end{proof}

\begin{lemma}\label{lemma:rwC2astC2}
Let $\Gamma = G_1 * G_2 * \cdots * G_m$ with $M = M(\Gamma)$ as in \cref{eq:mu_and_M_for_Gamma}. Then
\[
    \P(X_1 X_2 \cdots X_p \in \Gamma_{C_2 * C_2}^\tf) \leq 
    (p+1)^5 M^2 \cdot \norm{\lambda(x)}^p.
\]
\end{lemma}
\begin{proof}
Any element $\gamma \in \Gamma_{C_2 * C_2}^\tf$ can be written as a product of two involutions, each of them conjugate to an involution in a finite free factor of $\Gamma$. Thus, there exist $\alpha, \alpha', \beta, \beta' \in \Gamma$ such that
\[
    \gamma = \alpha^{-1} \alpha' \alpha \beta^{-1} \beta' \beta,
\]
where $\alpha', \beta'$ belong to finite free factors. We can now repeat the argument from the previous lemma, splitting the product $X_1 X_2 \cdots X_p$ into six parts according to $1 \leq t_1 \leq t_2 \leq \cdots \leq t_5 \leq p+1$, bounding the terms corresponding to $\alpha', \beta'$ by powers of $\norm{\lambda(x)}$, and using Cauchy-Schwarz to handle the sums over $\alpha, \beta$. The number of choices for $\alpha', \beta'$ is at most $M^2$, and the number of choices for the $t_i$ is at most $(p+1)^5$, yielding the desired bound.
\end{proof}

\begin{lemma}\label{lemma:rwInfinite}
Let $\Gamma = G_1 * G_2 * \cdots * G_m$. For any integer $d \geq 2$, define $\Gamma_{\pow,d}^{\tf}$ to be the set of elements in $\Gamma^{\tf}$ that are $d$-th powers of elements that are not proper powers. Then
\[
    \P(X_1 X_2 \cdots X_p \in \Gamma_{\pow,d}^{\tf}) \leq 
    (p+1)^4 \cdot \norm{\lambda(x)}^p.
\]
Moreover, the probability is nonzero only for $d \leq p \abs{x}$.
\end{lemma}
\begin{proof}

We start by proving the final statement. Note that, if $\Gamma^{\rm np}$ denotes the set of group elements that are not proper powers,
\[ \P(X_1 X_2 \cdots X_p \in \Gamma_{\pow,d}^{\tf}) = \sum_{\delta \in \Gamma^{\rm np}} \P(X_1 X_2 \cdots X_p = \delta^d) .\]
Hence, we focus on $\gamma = \delta^d$ for some $\delta \in \Gamma^{\rm np}$. The event $X_1 X_2 \cdots X_p = \gamma$ implies that $d \leq \abs{\delta^d} \leq \abs{X_1} + \abs{X_2} + \cdots + \abs{X_p} \leq p \abs{x}$, so the probability is nonzero only for $d \leq p \abs{x}$. From here on, we follow \cite[Proof of Lemma 3.10]{van2025strong} and \cite[Proof of Theorem 1.8]{magee2025strong}. Since $\Gamma$ admits normal forms, we can write $\delta = \alpha^{-1} \beta \alpha$ for some cyclically reduced word $\beta \in \Gamma$ and some $\alpha \in \Gamma$. Thus $\gamma = \alpha^{-1} \beta^d \alpha$. For each fixed $d$, we repeat the argument from the previous lemmas: we split the product $X_1 X_2 \cdots X_p$ into four parts according to $1 \leq t_1 \leq t_2 \leq t_3 \leq p+1$, bound the term corresponding to $\beta^d$ by $\norm{\lambda(x)}^{t_2-t_1}$, and use Cauchy-Schwarz to handle the sums over $\alpha$. The number of choices for the $t_i$ is at most $(p+1)^4$.
\end{proof}

\subsection{Temperedness of leading coefficients}
We conclude with the proof of Requirement~\ref{requirement:temperedness} for free products of finite groups.
\begin{theorem}
Let $\Gamma = G_1 * G_2 * \cdots * G_m$ with $\mu = \mu(\Gamma)$
as in \cref{eq:mu_and_M_for_Gamma}. Then, for all positive integers $k \leq \mu$, the functional $u_k$ is tempered at every self-adjoint element of $\C[\Gamma]$.
\end{theorem}
\begin{proof}
By \Cref{lem:positivization}, it suffices to prove the statement for generators of symmetric random walks on $\Gamma$. Let $x \in \C[\Gamma]$ be such a generator. Expand $x^p$ as in \cref{eq:xp_expectation}. Split this sum into two parts,
\[
    u_k(x^p) =
    \sum_{\gamma \in \Gamma^\tors} \P(X_1 X_2 \cdots X_p = \gamma) u_k(\gamma) +
    \sum_{\gamma \in \Gamma^\tf} \P(X_1 X_2 \cdots X_p = \gamma) u_k(\gamma).
\]
By Lemmas~\ref{lem:leading_coefficients_for_group_elements} and~\ref{lemma:rwTorsion}, the contribution of the first part is 
\[
    O_{\mu,M}(1) \cdot \P(X_1 X_2 \cdots X_p \in \Gamma^\tors) \\
    \leq
    O_{\mu,M}(p^{O(1)}) \cdot \norm{\lambda(x)}^p.
\]
Again using Lemma~\ref{lem:leading_coefficients_for_group_elements}, the second part only contributes when $k = \mu$, and it can be split further into three parts according to the partition of $\Gamma^\tf$ into $\Gamma_\pow^\tf$, $\Gamma_{C_2 * C_2}^\tf$, and $\Gamma_\infty^\tf$. In view of Lemma~\ref{lem:counting_C2C2_and_powers}, for $\gamma \in \Gamma_\infty^\tf$, we have $u_\mu(\gamma) = 0$, and hence this part does not contribute. Moreover, for $\gamma \in \Gamma_{C_2 * C_2}^\tf$, we have $u_\mu(\gamma) = 1$, and for $\gamma \in \Gamma_\pow^\tf$ with $\gamma = \delta^d$, we have $u_\mu(\gamma) \leq \omega(d) + \sigma(d) \leq O(d^2)$. By Lemma~\ref{lemma:rwInfinite}, the latter terms only contribute when $d \leq p \abs{x}$. Thus the contribution of the second part is at most
\[
    \P(X_1 X_2 \cdots X_p \in \Gamma_{C_2 * C_2}^\tf) +
    \sum_{2 \leq d \leq p |x|} O(d^2) \cdot
    \P(X_1 X_2 \cdots X_p \in \Gamma_{\pow,d}^\tf),
\]
which, again by the Lemmas~\ref{lemma:rwC2astC2} and~\ref{lemma:rwInfinite}, is bounded by
\[
    O_M(p^{O(1)}) \cdot \norm{\lambda(x)}^p +
    O((p |x|)^{O(1)}) \cdot \norm{\lambda(x)}^p.
\]
Taking $p$-th roots and limits as $p \to \infty$, we obtain the desired temperedness.
\end{proof}

\bibliographystyle{alpha}
\bibliography{biblio}

\end{document}